\documentclass[onefignum,onetabnum]{siamart220329}
\usepackage{graphics,graphicx,epstopdf,url}
\usepackage{arydshln}
\usepackage{color}
\usepackage{geometry}
\usepackage{tabularx}

\usepackage{multirow}
\usepackage{amsmath,amssymb,amsfonts}
\usepackage{mathrsfs}
\usepackage{xcolor}
\usepackage{textcomp}
\usepackage{manyfoot}
\usepackage{booktabs}

\usepackage{algorithmic}
\usepackage{url,cases,supertabular}
\usepackage{amssymb,mathtools}
\usepackage{enumerate,makecell}
\usepackage{amsfonts}
\usepackage{graphicx,epstopdf}
\usepackage{color,verbatim}
\usepackage{mathrsfs,subeqnarray,subfigure}
\usepackage{listings,array}
\usepackage{booktabs}
\usepackage{dashrule}
\usepackage{arydshln}
\usepackage{graphicx}
\usepackage{amsfonts}
\usepackage{mathrsfs}
\usepackage{graphicx}  \usepackage{epstopdf}
\usepackage{subfigure}  \usepackage{lineno}
\usepackage{multirow,bm}
\usepackage{color}
\usepackage{algorithm}
\usepackage{algorithmic}
\usepackage{longtable}
\usepackage{mdwtab}
\usepackage{mdwmath}
\usepackage{bbding}
\usepackage{booktabs}
\usepackage{rotating}
\usepackage{enumitem}
\usepackage{lineno}
\usepackage{url}
\usepackage{pdflscape}
\usepackage{makecell}

\numberwithin{equation}{section}

\newtheorem{example}{Example}[section]
\newtheorem{assumption}[theorem]{Assumption}
 \newtheorem{remark}[theorem]{Remark}

\newcommand{\ba}{\begin{array}}
\newcommand{\ea}{\end{array}}

\newcommand{\bit}{\begin{itemize}}
\newcommand{\eit}{\end{itemize}}
\newcommand{\be}{\begin{equation}}
\newcommand{\ee}{\end{equation}}
\newcommand{\bee}{\begin{equation*}}
\newcommand{\eee}{\end{equation*}}
\newcommand{\bea}{\begin{eqnarray}}
\newcommand{\eea}{\end{eqnarray}}

\newcommand{\Rmn}[1]{\uppercase\expandafter{\romannumeral#1}}
\renewcommand{\theenumi}{\roman{enumi}}

\numberwithin{equation}{section}

\numberwithin{theorem}{section}

\usepackage{lipsum}
\usepackage{clrscode}
\usepackage{appendix}
\usepackage{bm}
\usepackage{booktabs}
\usepackage{url}
\usepackage{multirow}
\usepackage{textcomp}
\usepackage{amsmath}
\usepackage{amsfonts}
\usepackage{amssymb}
\usepackage{mathrsfs}
\usepackage{hyperref}
\usepackage{graphicx,graphics,subfigure}
\usepackage{hyperref,url}
\usepackage{epsf,epstopdf}
\usepackage{algorithm}
\usepackage{algorithmic}
\usepackage{bbm}
\usepackage{dsfont}
\usepackage{indentfirst}
\usepackage{pdfpages}
\usepackage{pdflscape}
\usepackage{longtable}
\ifpdf
  \DeclareGraphicsExtensions{.eps,.pdf,.png,.jpg}
\else
  \DeclareGraphicsExtensions{.eps}
\fi
\allowdisplaybreaks[1]

\headers{Inexact Riemannian Gradient Descent Method}{Juan Zhou et al}

\title{Inexact Riemannian Gradient Descent Method for Nonconvex Optimization}

\author{
Juan Zhou\thanks{School of Mathematics and Computational Science, Xiangtan University, Xiangtan, 411105, China
(\email{juanzhou425@gmail.com}).}
\and Kangkang Deng\thanks{ Department of Mathematics,  National University of Defense Technology, Changsha, 410073,
China (\email{freedeng1208@gmail.com}).}
\and Hongxia Wang \thanks{Department of Mathematics,  National University of Defense Technology, Changsha, 410073,
China (\email{wanghongxia@nudt.edu.cn}).}
\and Zheng Peng \thanks{School of Mathematics and Computational Science, Xiangtan University, Xiangtan, 411105, China  (\email{pzheng@xtu.edu.cn}).}
}

\usepackage{amsopn}

\ifpdf
\hypersetup{
  pdftitle={Inexact Riemannian Gradient Descent Method},
  pdfauthor={Juan Zhou}
}
\fi

\allowdisplaybreaks[2]
\begin{document}

\maketitle

\begin{abstract}
Gradient descent methods are fundamental first-order optimization algorithms in both Euclidean spaces and Riemannian manifolds. However, the exact gradient is not readily available in many scenarios. This paper proposes a novel inexact Riemannian gradient descent algorithm for nonconvex problems, accompanied by a convergence guarantee. In particular, we establish two inexact gradient conditions on Riemannian manifolds for the first time, enabling precise gradient approximations. Our method demonstrates strong convergence results for both gradient sequences and function values. The global convergence with constructive convergence rates for the sequence of iterates is ensured under the Riemannian Kurdyka-\L ojasiewicz property. Furthermore, our algorithm encompasses two specific applications: Riemannian sharpness-aware minimization and Riemannian extragradient algorithm, both of which inherit the global convergence properties of the inexact gradient methods. Numerical experiments on low-rank matrix completion and principal component analysis problems validate the efficiency and practical relevance of the proposed approaches.

% We are concerned with decentralized optimization over a compact manifold, where the loss functions of local agents are defined by the local datasets. To jointly optimize the loss functions, the decentralized algorithms performs neighborhood average through a network graph. One limitation in the projection/retraction-type algorithms is that they require multi-step consensus to achieve the consensus and the optimality. Relaxing the multi-step consensus to single-step consensus is an open question due to the nonconvexity nature of manifold constraint. In this paper, we answer this question by carefully elaborating the  smoothness structure and the asymptotic 1-Lipschitz continuity associated with the manifold constraint. Our results reveal the empirical success of using single-step consensus and unit step size. Furthermore, we combine it with the communication compression strategy and propose a communication-efficient gradient algorithm for solving decentralized manifold optimization problem, where per-iteration communication cost is largely reduced. Moreover, the iteration complexity of $\mathcal{O}(1/K)$, which matches the Euclidean setting, is established. Numerical experiments demonstrate the efficiency of our proposed method compared with the state-of-the-art ones. 
\end{abstract}

% REQUIRED

\section{Introduction}\label{sec intro}
This paper focuses on the following smooth optimization problem on Riemannian manifolds:
\begin{equation}\label{prob}
    \min_{x \in \mathcal{M}} f(x),
\end{equation}
where $\mathcal{M}$ is a Riemannian submanifold embedded in $\mathbb{R}^n$, and $f: \mathcal{M} \rightarrow \mathbb{R}$ is a continuously differentiable ($\mathcal{C}^1$-smooth) and nonconvex function. One of the most classical and effective methods for solving \eqref{prob} is the Riemannian gradient descent (RGD) algorithm \cite{absil2008optimization,smith2014optimization}. Given an initial point $x_0\in \mathcal{M}$, the iterative procedure of the RGD is designed as follows:
\begin{equation}\label{eq:rgd}
x_{k+1}=\mathcal{R}_{x_k}(-t_k \mathrm{grad}f(x_k)),
\end{equation}
for all $k\in \mathbb{N}$, where $\mathcal{R}$ denotes the retraction operator on the manifold $\mathcal{M}$, $t_k\ge 0$ is the stepsize at the $k$-th iteration, and $\mathrm{grad} f$ is the Riemannian gradient of $f$, which will be defined in the next section. Due to its simplicity, the RGD method is widely applied to various optimization problems \cite{afsari2013convergence,wei2016guarantees,hou2020fast,huang2023gradient}. However, in many practical scenarios, the exact gradient $\mathrm{grad} f(x)$ may not be accessible, or deterministic errors can occur in the gradient computations. This compels us to develop inexact gradient-based algorithms. 

In the Euclidean setting, commonly used inexact gradient conditions are the unnormalized and normalized conditions \cite{pedregal2004introduction,carter1991global}, defined as follows:
\begin{align}
\|g - \nabla f(x)\| & \le \epsilon, \label{unnormalized} \\
\|g - \nabla f(x)\| & \le \nu \|\nabla f(x)\|,\label{normalized}
\end{align}
where $g$ is an approximation of $\nabla f(x)$, and $\epsilon, \nu \ge 0$ are parameters.   These two inexact conditions have broad applicability across various domains, including derivative-free optimization \cite{cartis2018global,berahas2022theoretical,khanh2023general}, finite difference approximation \cite{cartis2018global,paquette2020stochastic}, and more. Recently, \cite{khanh2023inexact} proposed an inexact gradient method under condition \eqref{unnormalized} for solving nonconvex smooth problems and provided a convergence analysis. Additionally, an inexact algorithm for nonsmooth convex problems has also been proposed by \cite{khanh2024new}. Although inexact algorithms have been developed in Euclidean spaces, no algorithm has yet been proposed for addressing problem \eqref{prob} under an inexact gradient oracle. The purpose of this paper is to generalize those two inexact conditions to Riemannian manifolds, propose a unified inexact gradient algorithmic framework under those conditions, and provide a comprehensive convergence analysis.

A typical example of gradient estimation on Riemannian manifolds involves using standard finite differences, linear interpolation, or other approximation techniques, as studied in derivative-free optimization methods \cite{li2023stochastic}. 

\begin{example}[zeroth-order optimization]\label{grad-free}
Gradient estimation on Riemannian manifolds can be performed using finite difference methods. Specifically, generate $u=Pu_o\in T_x\mathcal{M}$, where $u_0 \sim \mathcal{N}(0,I_n)$ in $\mathbb{R}^n$, and $P\in \mathbb{R}^{n\times n}$ is the orthogonal projection matrix onto $T_x\mathcal{M}$. The gradient estimation is then given by
$$
g_{\mu}(x)=\frac{f(\mathcal{R}_x(\mu u))-f(x)}{\mu}u=\frac{f(\mathcal{R}_x(\mu Pu_0))-f(x)}{\mu}Pu_o.
$$
It is important to note that $g_\mu (x)$ is a biased estimator of $\mathrm{grad} f(x)$. However, this method typically yields more stable approximations, as established in \cite{li2023stochastic}.  
\end{example}

Another intuitive example is the inexact gradient used in the Riemannian Sharpness-Aware Minimization (SAM) method \cite{yun2024riemannian}. 

\begin{example}[Riemannian SAM]\label{SAM}
In the context of inexact gradient conditions, the SAM method on Riemannian manifolds provides insights into convergence behavior. The approximate gradient 
% is defined as $$\textcolor{red}{g(x)=\mathcal{T}_{x_k^{adv}}^{x_k} \mathrm{grad} f(x_k^{adv})},$$ where the gradient 
is computed at the backward step, defined as
$$
x_k^{adv} =\mathcal{R}_{x_k}(\rho \frac{\mathrm{grad} f(x_k)}{\|\mathrm{grad} f(x_k)\|}),
$$
rather than at the current point $x_k$. This backward step is specifically designed to avoid minimizers with high sharpness, thereby enhancing generalization. As a result, this method facilitates average-iterate convergence analysis under a less aggressively decaying ascent learning rate, as discussed in \cite{yun2024riemannian}.
\end{example}

Motivated by the works of \cite{khanh2023inexact,khanh2024new,khanh2024globally}, we propose a unified algorithmic framework for solving problem \eqref{prob} under inexact gradient conditions. The main contributions of this paper are as follows:
\begin{itemize}
    \item We develop and rigorously justify a general framework for a novel inexact Riemannian gradient descent (IRGD) method and its standardized variants IRGDr. This framework allows us to establish strong convergence results, including the stationarity of accumulation points, the convergence of the sequence of gradients to the origin, the sequence of function values to the optimal value, and the sequence of the iterates to optimal solutions is ensured under the Riemannian Kurdyka-\L ojasiewicz (KL) property of the objective function, with convergence rates determined by the Riemannian KL exponent. To the best of our knowledge, this is the first IRGD method that provides guaranteed convergence results.
    
    \item Two specific applications within the IRGD framework are introduced: the Riemannian sharpness-aware minimization (RSAM) method under the unnormalized condition \eqref{unnormalized-r}, and the Riemannian extragradient (REG) method under the normalized condition \eqref{normalized-r}. These methods effectively apply the theoretical convergence analysis developed by IRGD, demonstrating the versatility and applicability of the framework in diverse optimization scenarios.
    
    \item Additionally, we provide empirical validation through numerical experiments on low-rank matrix completion (MC) and principal component analysis (PCA) problems. These experiments demonstrate the efficiency and practical applicability of the proposed inexact methods, including the REG method, highlighting their potential in optimization tasks.
\end{itemize}

The paper is organized as follows. Sect.~\ref{sec prelim} covers the preliminaries of the IRGD method. In Sect.~\ref{sec inexact RGD}, we introduce our unified algorithmic framework and present the convergence analysis under inexact gradient conditions. Applications of this framework are discussed and analyzed in Sect.~\ref{sec app}. Sect.~\ref{sec numerical} details the numerical experiments conducted on two Riemannian optimization problems to evaluate the performance of our proposed methods. Finally, Sect.~\ref{sec conclusion} offers concluding remarks and outlines potential future research directions.

\section{Preliminaries} \label{sec prelim}
The Riemannian concepts presented in this paper are consistent with the established Riemannian optimization literature \cite{absil2008optimization}. Let $\mathcal{M} \subset \mathbb{R}^n$ be a differentiable embedded submanifold, equipped with a smoothly varying inner product $\langle \cdot, \cdot \rangle_x: T_x\mathcal{M}\times T_x\mathcal{M}\rightarrow \mathbb{R}$ defined on the tangent space $\mathcal{M}$ at each point $x$. The norm of a vector $\xi_x\in T_x\mathcal{M}$ is then given by $\|\xi_x\|_x := \sqrt{\langle \xi_x, \xi_x\rangle_x}$.  The Riemannian gradient $\mathrm{grad} f(x) \in T_x\mathcal{M}$ of a smooth function $f$ at a point $x\in \mathcal{M}$ is the unique tangent vector satisfying $ \left< \mathrm{grad} f(x), \xi \right>_x = Df(x)[\xi]$, $ \forall \xi\in T_x\mathcal{M} ,$
where $Df(x)[\xi]$ represents the directional derivative along the direction $\xi_x$. 
As is well known, a point $\bar x$ is stationary point for a $\mathcal{C}^1$-smooth function $f\colon\mathcal{M}\rightarrow\mathbb{R}$ if $\mathrm{grad} f(\bar x)=0$. 

We next introduce the concepts of retraction and vector transport, which are defined as follows. A smooth mapping $R: T\mathcal{M} \rightarrow \mathcal{M}$ is called a retraction on a manifold $\mathcal{M}$ if its restriction at $x$, denoted as $\mathcal{R}_x: T_x\mathcal{M} \rightarrow \mathcal{M}$, satisfies
\begin{itemize}
    \item $\mathcal{R}_x(0_x) = x$ for all $x\in \mathcal{M}$, where $0_x$ denotes the zero element of $T_x\mathcal{M}$;
    \item the differential of $\mathcal{R}_x$ at $0_x$ is an
identity mapping on $T_x\mathcal{M}$, i.e, $D\mathcal{R}_x(0_x) = \mathrm{id}_{T_x\mathcal{M}}$.
\end{itemize}
A smooth mapping $T\mathcal{M}\oplus T\mathcal{M} \rightarrow T\mathcal{M}:(\eta_x,\xi_x)\mapsto \mathcal{T}_{\eta_x}(\xi_x)$
is called a vector transport on a manifold $\mathcal{M}$,
if it satisfies
\begin{itemize}
    \item $T_{\eta_x} \xi_x \in T_{\mathcal{R}_x(\eta_x)}\mathcal{M}$ for all $x\in \mathcal{M}$, and for all $\eta_x, \xi_x \in T_x\mathcal{M}$;
    \item $\mathcal{T}_{0_x}\xi_x = \xi_x$ for all $\xi_x\in T_x\mathcal{M}$;
    \item  $\mathcal{T}_{\eta_x}(a\xi_x + b\zeta_x)  = a \mathcal{T}_{\eta_x} \xi_x + b\mathcal{T}_{\eta_x} \zeta_x$ for all $a,b\in \mathbb{R}$, and for all $\eta_x, \xi_x, \zeta_x \in T_x\mathcal{M}$.
\end{itemize} 
A vector transport $\mathcal{T}$ is called isometric if it satisfies $\langle \mathcal{T}_{\eta_x}\xi_x, \mathcal{T}_{\eta_x}\zeta_x\rangle_{\mathcal{R}_x(\eta_x)}=\langle \xi_x, \zeta_x\rangle_x$, for any $\eta_x,\xi_x,\zeta_x\in T_x \mathcal{M}$.
The adjoint operator $\mathcal{T}^{\sharp}$ of a vector transport $\mathcal{T}$ is defined such that $\langle \xi_y,\mathcal{T}_{\eta_x} \zeta_x\rangle_y=\langle \mathcal{T}_{\eta_x}^{\sharp}\xi_y,\zeta_x\rangle_y$
for all $\eta_x, \zeta_x\in T_x\mathcal{M}$ and $\xi_y\in T_y\mathcal{M}$, where $y = \mathcal{R}_x(\eta_x)$. 
The inverse operator $\mathcal{T}^{-1}$ is defined by the condition $\mathcal{T}_{\eta_x}^{-1}\mathcal{T}_{\eta_x}=\mathrm{id}$ for all $\eta_x\in T_x\mathcal{M}$, where $\mathrm{id}$ denotes the identity operator.

The global convergence and convergence rates of various nonconvex algorithms benefit from the Riemannian KL property \cite{huang2022riemannian,bento2011convergence}. The following provides the definition.

\begin{definition} [Riemannian Kurdyka-\L ojasiewicz property] \rm \label{RKL}\rm A continuous smooth function $f:\mathcal{M}\rightarrow \mathbb{R}$ enjoys the Riemannian \textit{KL property} at $\bar x\in \mathrm{grad}f$ if and only if there exist $\eta\in (0,\infty]$, a neighborhood $U \subset \mathcal{M}$ of $\bar x$, and a desingularizing concave continuous function $\varphi:[0,\delta)\rightarrow[0,\infty)$ such that:
\begin{enumerate}[label=(\roman*)]
    \item $\varphi(0)=0$.
    \item $\varphi$ is $\mathcal{C}^1$-smooth on $(0,\delta)$.
    \item $\varphi^\prime>0$ on $(0,\delta)$.
    \item For all $x\in U$ with $f(\bar x)<f(x)<f(\bar x)+\delta$, we have 
    \begin{align}\label{KL 2}
    \varphi^\prime(f(x)-f(\bar x))\|\mathrm{grad} f(x)\|\ge 1.
    \end{align}
\end{enumerate}
\end{definition}

The following lemma establishes that if the Riemannian KL property holds at every single point within a compact set that shares the same function value, then there exists a single desingularizing function such that the Riemannian KL property holds uniformly for all points within that compact set. This generalization is also implicitly covered in the proofs presented in \cite{hosseini2015convergence,huang2022riemannian}.

\begin{lemma}\label{RKL:compact}
    Let $\bar \Omega$ be a compact set in $\mathcal{M}$ and let $h:\mathcal{M}\rightarrow (-\infty, \infty]$ be a continuous function. Assume that $h$ is a constant on $\bar \Omega$ and satisfies the Riemannian KL property at each point of $\bar \Omega$. Then, there exists $\bar \omega, \delta>0$ and a continuous concave function $\varphi:[0, \delta]\rightarrow [0, \infty)$ such that for all $\bar u$ in $\bar \Omega$ and all $u$ in the following intersection:
    $$\{u\in \mathcal{M}:\inf_{v\in \bar \Omega}\|u-v\|< \bar \omega\}\cap \{u\in \mathcal{M}:h(\bar u)<h(u)<h(\bar u)+\delta\},$$
    one has 
    $$\varphi^\prime(h(u)-h(\bar u))\|\mathrm{grad} h(u)\|\ge 1.$$
\end{lemma}

\section{Inexact Riemannian gradient descent methods} \label{sec inexact RGD}
In this section, we introduce a unified algorithmic framework for the IRGD method, featuring two inexact gradient conditions designed to solve problem \eqref{prob}. The IRGD techniques are defined as iterative optimization schemes, given by
\begin{equation}\label{eq:rcg-xk+1}
x_{k+1}=\mathcal{R}_{x_k}(-t_k g_k),
\end{equation}
where $t_k$ is a diminishing stepsize, and $g_k$ is an approximation of $\mathrm{grad} f(x_k)$ that satisfies
\begin{equation}\label{unnormalized-r}
\|g_k - \mathrm{grad} f(x_k)\| \le \epsilon_k,
\end{equation}
or
\begin{equation}\label{normalized-r}
\|g_k - \mathrm{grad} f(x_k)\| \le \nu \|\mathrm{grad} f(x_k)\|,
\end{equation}
where $\epsilon_k > 0$ and $\nu \ge 0$ is the relative error parameter. The main motivation for the construction algorithm is that by making the inexact gradient $g_k$, it avoids the situation where the gradient is not available, which ensures the convergence of the algorithm. 

Next, we recall some basic stepsize selections for the iterative procedure \eqref{eq:rcg-xk+1}. If the step size $t_k$ satisfies the Armijo rule, it guarantees the nonincreasing property of the sequence $\{f(x_k)\}$. However, the Armijo stepsize may be very small, leading to many iterations with only minor changes in the sequence. Although a constant step size can significantly simplify the iterative design, it does not generally guarantee the nonincreasing property of $\{f(x_k)\}$, resulting in inefficiency \cite{bertsekas1997nonlinear,nocedal1999numerical}. This observation motivated our approach, which is based on the analysis of a diminishing stepsize given by 
\begin{align}\label{diminishing}
   \sum_{k=1}^\infty t_k=\infty\;\text{ and }\;\sum_{k=1}^\infty t_k^2<\infty.
\end{align}
This step size also ensures that $t_k\downarrow 0$, which is satisfied for the cosine step size scheduler in each cycle \cite{loshchilov2016sgdr}, making it a more favorable approach.

Before giving our algorithms, the following key lemma provides a unified conclusion for the convergence analysis of our methods. 
\begin{lemma}\label{R-dist}\cite{boumal2019global}
   Assume that there exist a constant $\kappa>0$ such that 
    \begin{align*}
    \|\mathcal{R}_x(\eta)-x\|\le \kappa \|\eta\|,
    \end{align*}
    for all $x\in \mathcal{M}$ and $\eta\in T_x\mathcal{M}$.
\end{lemma}

% [three sequences lemma]
\begin{lemma}\cite{khanh2024fundamental} \label{three sequences lemma}
   Let $\{\alpha_k\},\{\beta_k\}, \{\gamma_k\}$ be sequences of nonnegative numbers satisfying the conditions
   \begin{align}
       &\alpha_{k+1}-\alpha_k\le  \beta_k\alpha_k+ \gamma_k\;\text{ for sufficient large }\;k\in\mathbb{N},\label{a}\\
       &\{\beta_k\}\text{ is bounded}, \sum_{k=1}^\infty \beta_k=\infty,\; \sum_{k=1}^\infty \gamma_k<\infty,\;\mbox{ and }\;\sum_{k=1}^\infty \beta_k\alpha_k^2<\infty.\label{b}
   \end{align}
Then we have that $\alpha_k\rightarrow0$ as $k\to\infty$.
\end{lemma}

\begin{lemma}\cite{khanh2022inexact}\label{stationary point lemma}
Let $\{x_k\}$ and $\{\eta_k\}$ be sequences in $\mathcal{M}$ satisfying the condition
\begin{align}\label{series is finite}
\sum_{k=1}^\infty \|x_{k+1}-x_k\|\cdot\|\eta_k\|<\infty.
\end{align}
If $\bar x$ is an accumulation point of the sequence $\{x_k\}$ and $0$ is an accumulation points of the sequence $\{\eta_k\}$, then there exists an infinite set $J\subset\mathbb{N}$ such that we have
\begin{align}\label{relation dk xk 3}
x_k\overset{J}{\rightarrow}\bar x\;\mbox{ and }\;\eta_k\overset{J}{\rightarrow}0.
\end{align}
\end{lemma}

\begin{remark}
    We use $\mathbb{N}:=\{1,2,\ldots\}$ signifies the collection of natural numbers. The symbol $x_k\xrightarrow{J}\bar x$ means that $x_k\to\bar x$ as $k\to\infty$ with $k\in J\subset\mathbb{N}$.
\end{remark}

\subsection{Unnormalized condition}\label{cond-I}
In this subsection, we present a general framework for our novel IRGD method, focusing on cases where the inexact gradient satisfies the unnormalized condition given by \eqref{unnormalized-r}. We also provide a detailed convergence analysis of the IRGD method.

\begin{algorithm}[H]
\caption{Inexact Riemannian Gradient Descent (IRGD) Methods }\label{Alg-RGD}
\begin{algorithmic}[1]
\STATE Choose some initial point $x_0\in\mathcal{M},$ sequence of errors $\{\epsilon_k\}\subset (0,\infty)$, and sequence of stepsizes $\{t_k\}\subset(0,\infty).$ For $k=1,2,\ldots,$ do the following
\STATE Set $x_{k+1}=\mathcal{R}_{x_k}(-t_k g_k)$ with $\|g_k-\mathrm{grad} f(x_k)\|\le \epsilon_k$, where $g_k\in T_{x_k} \mathcal{M}$.
\end{algorithmic}
\end{algorithm}

The algorithm proposed in this work is inspired by the inexact gradient descent methods studied in Euclidean space by \cite{khanh2023inexact, khanh2024new, khanh2024fundamental}. The latter methods consider relative errors in gradient calculation, while Algorithm \ref{Alg-RGD} uses absolute errors. This approach is particularly suitable for the constructions of RSAM, as established in Sect.~\ref{app:rsam}. The following assumption presents the Riemannian setting of $L$-smooth functions.

\begin{assumption}[$L$-Retraction Smoothness]\label{descent condition}
Assume that there exists a constant $L>0$ such that
\begin{align*}
f(\mathcal{R}_x(\eta))\le f(x)+\langle \mathrm{grad} f(x),\eta\rangle +\frac{L}{2}\|\eta\|^2,
\end{align*}
for all $x\in\mathcal{M}, \eta\in T_x \mathcal{M}$.
\end{assumption}
The following lemma immediately confirms that IRGD is a descent algorithm.
\begin{lemma}\label{descent lemma}
    Suppose that Assumption \ref{descent condition} holds. Let $\{x_k\}$ be generated by Algorithm~\ref{Alg-RGD} with stepsizes and errors satisfying the conditions
   \begin{align}\label{parameter general}
        \sum_{k=1}^\infty t_k=\infty,\;t_k\downarrow0, \sum_{k=1}^\infty t_k \epsilon_k<\infty,\;\limsup \epsilon_k<2.
    \end{align}
    Then there exists $K>0$ such that for $k>K$,
    $$
    f(\mathcal{R}_{x_k}(-t_kg_k))\le f(x_k) -c_1 t_k\|\mathrm{grad} f(x_k)\|^2+c_2t_k\epsilon_k.
    $$    
\end{lemma}

\begin{proof}
    First, fix $k\in\mathbb{N}$ and deduce from the Cauchy-Schwarz inequality that
    \begin{align}\label{esti product}
    \langle g_k,\mathrm{grad} f(x_k)\rangle&=\langle g_k-\mathrm{grad} f(x_k),\mathrm{grad} f(x_k)\rangle+\|\mathrm{grad} f(x_k)\|^2\nonumber\\
    &\ge -\|g_k-\mathrm{grad} f(x_k)\|\cdot\|\mathrm{grad} f(x_k)\|+\|\mathrm{grad} f(x_k)\|^2\nonumber\\
    &\ge -\epsilon_k\|\mathrm{grad} f(x_k)\|+\|\mathrm{grad} f(x_k)\|^2.
    \end{align}
By \eqref{parameter general}, we find some $c_1>0,c_2\in(0,1)$, and $K\in\mathbb{N}$ such that 
\begin{align}\label{defi c1c2}
    \frac{1}{2}( 2-Lt_k- \epsilon_k+L t_k\epsilon_k)\ge c_1,\quad \frac{1}{2}(1-Lt_k)+\frac{Lt_k\epsilon_k}{2}\le  c_2,\;\mbox{ and}\quad Lt_k<1, \;
\end{align}
for all $k\ge K$. Since $\mathrm{grad} f$ is Lipschitz continuous with constant $L$, it follows from the descent condition in Assumption \ref{descent condition} and the estimate \eqref{esti product} that
    \begin{align*}
        &f(\mathcal{R}_{x_k}(-t_kg_k))\le f(x_k)-t_k\langle\mathrm{grad} f(x_k),g_k\rangle+\frac{Lt_k^2}{2}\|g_k\|^2\nonumber\\
        &=f(x_k)-t_k(1-Lt_k)\langle\mathrm{grad} f(x_k),g_k\rangle+\frac{Lt_k^2}{2}(\|g_k-\mathrm{grad} f(x_k)\|^2-\|\mathrm{grad} f(x_k)\|^2)\nonumber\\
        &\le f(x_k)-t_k(1-Lt_k)\Big(-\epsilon_k\|\mathrm{grad} f(x_k)+\|\mathrm{grad} f(x_k)\|^2\Big)+\frac{Lt_k^2\epsilon_k^2}{2}-\frac{Lt_k^2}{2}\|\mathrm{grad} f(x_k)\|^2\nonumber\\
        &=f(x_k)-\frac{t_k}{2}(2-{Lt_k})\|\mathrm{grad} f(x_k)\|^2+t_k(1-Lt_k)\epsilon_k\|\mathrm{grad} f(x_k)\|+\frac{Lt_k^2\epsilon_k^2}{2}\nonumber\\
        &\le f(x_k)-\frac{t_k}{2}(2-{Lt_k})\|\mathrm{grad} f(x_k)\|^2+\frac{1}{2} t_k(1-Lt_k)\epsilon_k\Big(1+\|\mathrm{grad} f(x_k)\|^2\Big)+\frac{Lt_k^2\epsilon_k^2}{2}\nonumber\\        
        &=f(x_k)-\frac{t_k}{2}(2-Lt_k-\epsilon_k+L t_k\epsilon_k)\|\mathrm{grad} f(x_k)\|^2+\frac{1}{2} t_k\epsilon_k(1-Lt_k)+\frac{Lt_k^2\epsilon_k^2}{2}\nonumber\\
        &=f(x_k)-\frac{t_k}{2}(2-Lt_k-\epsilon_k+L t_k\epsilon_k)\|\mathrm{grad} f(x_k)\|^2+t_k\epsilon_k\Big(\frac{1}{2} (1-Lt_k)+\frac{Lt_k\epsilon_k}{2}\Big).\nonumber 
    \end{align*} 
Combining this with \eqref{defi c1c2} gives us
\begin{align}\label{non descent}
    f(\mathcal{R}_{x_k}(-t_kg_k))\le f(x_k) -c_1 t_k\|\mathrm{grad} f(x_k)\|^2+c_2t_k\epsilon_k\;\text{ whenever }\;k\ge K.
\end{align}
\end{proof}

\begin{assumption}\label{lower bounded}
    Assume that the objective function $f$ is bounded below, i.e., $\inf_{k\in\mathbb{N}} f(x_k)>-\infty$.
\end{assumption}

To analyze the convergence properties of the IRGD method based on the Riemannian KL property, we require a result analogous to those in \cite{huang2022riemannian} concerning vector transport, under the following assumption.

\begin{assumption}[Individual Retraction Lipschitzness]\label{Lips def}
    A function $f:\mathcal{M}\rightarrow \mathbb{R}$ is said to posses a Lipschitz continuous gradient with the uniform constant $L>0$, or equivalently it belongs to the class $\mathcal{C}^{1,L}$, if we have the estimate
\begin{align*}
\|\mathcal{T}_{y}^x\mathrm{grad} f(y)-\mathrm{grad} f(x)\|\le L\|\eta\|, ~~\text{or}~~\|\mathrm{grad} f(y)-\mathcal{T}_{\mathcal{R}_x(\eta)}^{-\sharp}\mathrm{grad} f(x)\|\le L\|\eta\|,
\end{align*}
for all $x, y=\mathcal{R}_x(\eta)\in\mathcal{M}$.
\end{assumption}

The global convergence properties of Algorithm \ref{Alg-RGD} are detailed in the following theorem, which addresses both the gradient sequences and the function values.

\begin{theorem}\label{convergence Alg-RGD}
   Suppose that Assumption \ref{descent condition},\ref{lower bounded} and \ref{Lips def} holds.  Let $\{x_k\}$ be generated by Algorithm~\ref{Alg-RGD} with stepsize and errors satisfying the conditions \eqref{parameter general}.
   % \begin{align}\label{parameter general}
   %      \sum_{k=1}^\infty t_k=\infty,\;t_k\downarrow0, \sum_{k=1}^\infty t_k \epsilon_k<\infty,\;\limsup \epsilon_k<2.
   %  \end{align}
    Then the following convergence properties hold:
    
    \begin{enumerate}[label=(\roman*)]
    \item\label{thm1-grad} $\mathrm{grad} f(x_k)\rightarrow0$, and thus every accumulation point of the iterative sequence $\{x_k\}$ is stationary for $f.$
    \item\label{thm1-f(x_k)} If $\bar x$ is an accumulation point of the sequence $\{x_k\}$, then  $f(x_k)\rightarrow f(\bar x)$.
    \end{enumerate}
\end{theorem}

\begin{proof}
Defining $u_k:=c_2\sum_{i=k}^\infty t_i\epsilon_i$ for $k\in\mathbb{N}$, we get that $u_k\rightarrow 0$ as $k\rightarrow\infty$ and $u_k-u_{k+1}=c_2t_k\epsilon_k$ for all $k\in\mathbb{N}.$ Then \eqref{non descent} can be rewritten as
    \begin{align}\label{descent f xk uk}
        f(\mathcal{R}_{x_k}(-t_kg_k))+u_{k+1}\le f(x_k)+u_k-c_1t_k\|\mathrm{grad} f(x_k)\|^2,\quad k\ge K.
    \end{align}
    
To proceed now with the proof of \ref{thm1-grad}, we deduce from \eqref{descent f xk uk} combined with $\inf_{k\in\mathbb{N}} f(x_k)>-\infty$ and $u_k\rightarrow0$ as $k\rightarrow\infty$ that
    \begin{align*}
        c_1\sum_{k=K}^\infty t_k\|\mathrm{grad} f(x_k)\|^2&\le \sum_{k=K}^\infty (f(x_k)-f(\mathcal{R}_{x_k}(-t_kg_k))+u_k-u_{k+1} )\\
        &\le f(x_K)-\inf_{k\in\mathbb{N}} f(x_k)+u_K<\infty.
    \end{align*}
Next we employ Lemma~\ref{three sequences lemma} with $\alpha_k:=\|\mathrm{grad} f(x_k)\|$, $\beta_k:=Lt_k$, and $\gamma_k:=Lt_k\epsilon_k$ for all $k\in\mathbb{N}$ to derive $\mathrm{grad} f(x_k)\rightarrow0.$ Observe first that condition \eqref{a} is satisfied due to the the estimates
\begin{align*}
    {\alpha_{k+1}-\alpha_k}&={\|\mathrm{grad} f(x_{k+1})}\|-\|\mathrm{grad} f(x_k)\| \le \|\mathcal{T}_{x_{k+1}}^{x_k}\mathrm{grad} f(x_{k+1})-\mathrm{grad} f(x_k)\|\\
    &=Lt_k\|g_k\| \le Lt_k (\|\mathrm{grad} f(x_k)\|+\|g_k-\mathrm{grad} f(x_k)\|)\\
    &\le Lt_k (\|\mathrm{grad} f(x_k)\|+\epsilon_k)=\beta_k \alpha_k+\gamma_k\;\text{ for all }\;k\in\mathbb{N}.
\end{align*}
Further, the conditions in \eqref{b} hold by \eqref{parameter general} and $\sum_{k=1}^\infty t_k\|\mathrm{grad} f(x_k)\|^2<\infty.$ As all the assumptions \eqref{a}, \eqref{b} are satisfied, Lemma~\ref{three sequences lemma} tells us that $\|\mathrm{grad} f(x_k)\|=\alpha_k\rightarrow0$ as $k\rightarrow\infty.$

To verify \ref{thm1-f(x_k)}, deduce from \eqref{descent f xk uk} that $\{f(x_k)+u_k\}$ is nonincreasing. Since $\inf_{k\in\mathbb{N}} f(x_k)>-\infty$ and $u_k\rightarrow 0$, it follows that $\{f(x_k)+u_k\}$ is bounded from below, and thus is convergent. Taking into account that $u_k\rightarrow 0$, it follows that $f(x_k)$ is convergent as well. Since $\bar x$ is an accumulation point of $\{x_k\}$,
the continuity of $f$ tells us that $f(\bar x)$ is also an accumulation point of $\{f(x_k)\},$ which immediately yields $f(x_k)\rightarrow f(\bar x)$ due to the convergence of $\{f(x_k)\}.$ 
\end{proof}

The following assumption on the desingularizing function, as defined in Definition \ref{RKL}, is utilized by \cite{li2023convergence}.

\begin{assumption}\label{assu desi}
    There is some $C>0$ such that 
    \begin{align*}
    C[\varphi^\prime(x+y)]^{-1}\le (\varphi^\prime(x))^{-1}+(\varphi^\prime(y))^{-1},
    \end{align*}
    whenever $x,y\in (0,\delta)$ with $x+y<\delta$.
\end{assumption}

When $f$ satisfies the Riemannian KL property in the set of all accumulation points, the following convergence result holds for the sequence of iterates.

\begin{theorem}\label{convergence iterates}
     Under the same condition as in the Theorem \ref{convergence Alg-RGD} and $\mathcal{S}$ denote the set of all accumulation points. Suppose that $f$ satisfies the Riemannian KL property at every point in $\mathcal{S}$ with the desingularizing function $\varphi$ satisfying Assumption~\ref{assu desi}. Assume in addition that 
    \begin{align}\label{desing condi}
            \sum_{k=1}^\infty t_k\Big(\varphi^\prime(\sum_{i=k}^\infty t_k\epsilon_k)\Big)^{-1}<\infty,
    \end{align} 
    and that $f(x_k) > f(\bar x)$ for sufficiently large $k \in \mathbb{N}$. Then $x_k \rightarrow \bar x$ as $k\rightarrow\infty$. In particular, if $\bar x$ is a global minimizer of $f$, then either $f(x_k) = f(\bar x)$ for some $k \in \mathbb{N}$, or $x_k \rightarrow \bar x$. 
\end{theorem}

\begin{proof}
First note that the global convergence result in \ref{thm1-grad} of Theorem \ref{convergence Alg-RGD} implies that every point in $\mathcal{S}$ is a stationary point. Since $\mathrm{grad} f(x_k)\rightarrow0$, there exists a $\delta >0$ such that $\|\mathrm{grad} f(x_k)\|\le \delta$ for all $k$. Thus, the application of Lemma \ref{R-dist} implies that
\begin{align}\label{Ostrowski condition:1}
    \|x_{k+1}-x_k\|
    & =\|\mathcal{R}_{x_k}(t_kg_k)-x_k\|\le \kappa t_k\|g_k\| \nonumber\\
    & \le \kappa t_k(\|g_k-\mathrm{grad}f(x_k)\|+\|\mathrm{grad} f(x_k)\|) \nonumber\\
    &\le \kappa t_k(\epsilon_k +\|\mathrm{grad} f(x_k)\|)\rightarrow 0.
\end{align}
Then by \cite{bolte2014proximal}, we konw that $\mathcal{S}$ is a compect set. Morever, since $f(x_k)$ is convergent, thus $f$ has the same value at all the points in $\mathcal{S}$. Therefore, by Lemma \ref{RKL:compact}, there exists a single desingularizing function satisfying Assumption~\ref{assu desi}, denote $\varphi$, for the Riemannian KL property of $f$ to hold at all the points in $\mathcal{S}$. 

In the case when $f(x_k) > f(\bar x)$ for all $k$, since $f(x_k)\rightarrow f(\bar x)$, $ \inf_{\bar x \in \mathcal{S}}\|x_k-\bar x\|\rightarrow 0$, by the Riemannian KL property of $f$ on $\mathcal{S}$, there exists an $l>0$ such that 
\begin{align}\label{for eq3}
    \varphi^\prime(f(x_k)-f(\bar x))\|\mathrm{grad} f(x_k)\|\ge 1 \;\text{ for all }\;k\ge l.
\end{align}
Define $\Delta_{p,q}:=\varphi (f(x_p)-f(\bar x)+u_p)- \varphi (f(x_q)-f(\bar x)+u_q)$ for all $p, q\in \mathbb{N}$, and combining this with $u_k>0$ and $f(x_k)-f(\bar x)>0$ gives us
\begin{subequations}
\begin{align}
\Delta_{k,k+1}
        &\ge \varphi^\prime(f(x_k)-f(\bar x)+u_k)(f(x_k)+u_k-f(x_{k+1})-u_{k+1})\label{eq1}\\
        &\ge \frac{C}{(\varphi^\prime(f(x_k)-f(\bar x)))^{-1}+(\varphi^\prime(u_k))^{-1}}c_1t_k\|\mathrm{grad} f(x_k)\|^2 \label{eq2}\\
        &\ge \frac{C}{\|\mathrm{grad} f(x_k)\|+(\varphi^\prime(u_k))^{-1}}c_1t_k\|\mathrm{grad} f(x_k)\|^2, \label{eq3}
\end{align}
\end{subequations}
where \eqref{eq1} follows from the concavity of $\varphi$, \eqref{eq2} follows from \eqref{descent f xk uk} and Assumption~\ref{assu desi}, and \eqref{eq3} follows from \eqref{for eq3}. Taking the square root of both sides in \eqref{eq3} and employing the AM-GM inequality yield
    \begin{align}\label{eq5}
      t_k\|\mathrm{grad} f(x_k)\|& = \sqrt{t_k}\cdot \sqrt{t_k\|\mathrm{grad} f(x_k)\|^2}\\
      &\le \sqrt{\frac{1}{Cc_1}(\Delta_{k,k+1})t_k(\|\mathrm{grad} f(x_k)\|+(\varphi^\prime(u_k))^{-1})}\nonumber\\
      &\le \frac{1}{2Cc_1}(\Delta_{k,k+1})+\frac{1}{2}t_k\Big((\varphi^\prime(u_k))^{-1}+\|\mathrm{grad} f(x_k)\|\Big).
    \end{align}
Using the nonincreasing property of $\varphi^\prime$ due to the concavity of $\varphi$ and the choice of $c_2\in (0,1)$ ensures that 
\begin{align*}
\Big(\varphi^\prime(u_k)\Big)^{-1}=\Big(\varphi^\prime(c_2\sum_{i=k}^\infty t_i\epsilon_i)\Big)^{-1}\le \Big(\varphi^\prime(\sum_{i=k}^\infty t_i\epsilon_i)\Big)^{-1}.
\end{align*}
Rearranging terms and taking the sum over $i=l+1, \ldots, k$ of inequality \eqref{eq5} gives us
    \begin{align*}
        \sum_{i={l+1}}^{k} t_i\|\mathrm{grad} f(x_i)\|&\le \frac{1}{{Cc_1}} \sum_{i={l+1}}^{k}(\Delta_{i,i+1})+ \sum_{i={l+1}}^{k} t_i\Big(\varphi^\prime(u_{i})\Big)^{-1}\\
        &= \frac{1}{{Cc_1}}(\Delta_{l+1,k+1})+ \sum_{i={l+1}}^{k} t_i \Big(\varphi^\prime (c_2\sum_{i=k}^\infty t_i\epsilon_{i})\Big)^{-1}\\
        &\le \frac{1}{{Cc_1}} \Delta_{l+1} + \sum_{i={l+1}}^{k} t_i\Big(\varphi^\prime(\sum_{i=k}^\infty t_i\epsilon_i)\Big)^{-1}.
    \end{align*}
Taking the number $C$ from  Assumption~\ref{assu desi}, remembering that $\bar x$ is an accumulation point of $\{x_k\},$ and using $f(x_k)+u_k\downarrow f(\bar x)$, $\Delta_k\downarrow0$  as $k\rightarrow\infty$ together with  condition \eqref{desing condi}, which yields $ \sum_{k=1}^\infty t_k\|\mathrm{grad} f(x_k)\|<\infty$. Inserting this into \eqref{Ostrowski condition:1} gives 
\begin{align*}
      \sum_{k=1}^\infty \|x_{k+1}-x_k\| \le \kappa \sum_{k=1}^\infty t_k\epsilon_{k} +\kappa\sum_{k=1}^\infty t_k\|\mathrm{grad} f(x_k)\|<\infty,
\end{align*}
which justifies the convergence of $\{x_k\}$ and thus completes the proof of the theorem.
\end{proof}

\subsection{Normalized condition}\label{cond-II}
In this subsection, we present a general framework for the standardized variants of our novel IRGD method with relative error and analyze the convergence properties of IRGDr under the inexact gradient normalized condition given by \eqref{normalized-r}.

\begin{algorithm}
\caption{IRGDr method} \label{Alg-RGDr}
\begin{algorithmic}[1]
\STATE
        Choose some initial point $x_0\in\mathcal{M},$ a relative error $\nu\ge 0$, and a sequence of stepsizes $\{t_k\}\subset[0,\infty).$ For $k=1,2,\ldots,$ do the following
\STATE
        Set $x_{k+1}=\mathcal{R}_{x_k}(-t_k g_k)$ with $\|g_k-\mathrm{grad} f(x_k)\| \le \nu \|\mathrm{grad} f(x_k)\|$, where $g_k\in T_{x_k} \mathcal{M}$.
\end{algorithmic}
\end{algorithm}

This algorithm differs from Algorithm \ref{Alg-RGD} in that the convergence properties of IRGDr are established under stronger stepsize assumptions and provide additional conclusions on the convergence rate. This makes it particularly suitable for constructing the REG method discussed in Sect.~\ref{app:reg}. The following lemma verifies the descent property of the objective function.

\begin{lemma}
    Let $\{x_k\}$ be the sequence generated by Algorithm~\ref{Alg-RGDr}. It holds that
    $$
    (1-\nu )\|\mathrm{grad} f(x_k)\|\le \|g_k\|\le(1+\nu) \|\mathrm{grad} f(x_k)\|.
    $$
\end{lemma}

\begin{proof}
    Using $\|\mathrm{grad} f(x_k)-g_k\|\le \nu\|\mathrm{grad} f(x_k)\|$ gives us the estimates
\begin{align}\label{es 1}
    \|g_k\|^2&=\|\mathrm{grad} f(x_k)-g_k\|^2 -\|\mathrm{grad} f(x_k)\|^2+2\langle\mathrm{grad} f(x_k),g_k\rangle \nonumber\\
    &\le \nu^2 \|\mathrm{grad} f(x_k)\|^2-\|\mathrm{grad} f(x_k)\|^2+2\langle\mathrm{grad} f(x_k),g_k\rangle\nonumber\\
    &= -(1-\nu^2)\|\mathrm{grad} f(x_k)\|^2+2\langle\mathrm{grad} f(x_k),g_k\rangle,
\end{align}
\begin{align}\label{es 2}
    \langle\mathrm{grad} f(x_k),g_k\rangle &=\langle\mathrm{grad} f(x_k),g_k-\mathrm{grad} f(x_k)\rangle+\|\mathrm{grad} f(x_k)\|^2\nonumber\\
    &\le \|\mathrm{grad} f(x_k)\|\cdot\|g_k-\mathrm{grad} f(x_k)\|+\|\mathrm{grad} f(x_k)\|^2\nonumber\\
    &\le(1+\nu)\|\mathrm{grad} f(x_k)\|^2,
\end{align}
\begin{align}\label{es 3}
    -\langle\mathrm{grad} f(x_k),g_k\rangle &=-\langle\mathrm{grad} f(x_k),g_k-\mathrm{grad} f(x_k)\rangle-\|\mathrm{grad} f(x_k)\|^2\nonumber\\
    &\le \|\mathrm{grad} f(x_k)\|\cdot\|g_k-\mathrm{grad} f(x_k)\|-\|\mathrm{grad} f(x_k)\|^2\nonumber\\
    &\le -(1-\nu)\|\mathrm{grad} f(x_k)\|^2,
\end{align}
\begin{align}
 \|\mathrm{grad} f(x_k)\|-\|g_k-\mathrm{grad} f(x_k)\|  \le \|g_k\|\le \|\mathrm{grad} f(x_k)\|+\|g_k-\mathrm{grad} f(x_k)\|,
\end{align}
which in turn implies that
\begin{align}\label{two inequality}
    (1-\nu )\|\mathrm{grad} f(x_k)\|\le \|g_k\|\le(1+\nu) \|\mathrm{grad} f(x_k)\|\text{ for all }k\in\mathbb{N}.
\end{align}
\end{proof}

Based on the assumption on the stepsize, we can derive the fundamental convergence properties of Algorithm \ref{Alg-RGDr}. These properties include the stationarity of accumulation points, the convergence of the gradient sequence to the origin, and the convergence of the function values to the optimal value. These results are presented in the following theorem.

\begin{theorem}\label{theo Alg-RGDr}
    Suppose that Assumption \ref{descent condition}, \ref{lower bounded}, \ref{Lips def} holds. Let $f:\mathcal{M}\rightarrow\mathbb{R}$ be a smooth function satisfying the descent condition for some constant $L>0,$ and let $\{x_k\}$ be the sequence generated by Algorithm~\ref{Alg-RGDr} with  the relative error $\nu \in [0,1)$, and the stepsizes satisfying 
    \begin{align}\label{stepsize Alg-RGDr}
       \sum_{k=1}^\infty t_k=\infty\;\text{ and }\; t_k\in [0,\frac{2-2\nu-\delta}{L(1+\nu)^2}]
    \end{align}
     for sufficiently large $k\in \mathbb{N}$ and for some $\delta>0$. Then either $f(x_k)\rightarrow-\infty$, or we have the assertions:
    \begin{enumerate}[label=(\roman*)]
    \item\label{thm3-stationary point} Every accumulation point of $\{x_k\}$ is a stationary point of the cost function $f$.
    \item\label{thm3-f(x_k)} If the sequence $\{x_k\}$ has any accumulation point $\bar x$, then $f(x_k)\downarrow f(\bar x)$.
    \item\label{thm3-grad} If $f\in\mathcal{C}^{1,L}$, then $\mathrm{grad} f(x_k)\rightarrow 0.$ 
    \end{enumerate}
\end{theorem}
    
\begin{proof}
Using condition \eqref{stepsize Alg-RGDr}, we find $N\in\mathbb{N}$ so that $2-2\nu -Lt_k(1+\nu)^2 \ge \delta$ for all $k\ge N.$ Select such a natural number $k$ and use the Lipschitz continuity of $\mathrm{grad} f$ with constant $L$ to deduce from Assumption \ref{descent condition}, the relationship $x_{k+1}=\mathcal{R}_{x_k}(-t_k g_k)$, the estimates \eqref{es 1}--\eqref{es 3} and \eqref{two inequality} that
\begin{align}
&f(x_{k+1})\le f(x_k)-t_k\langle\mathrm{grad}f(x_k),g_k\rangle+\frac{Lt_k^2}{2} \|g_k\|^2\nonumber\\
    &\le f(x_k)-t_k\langle\mathrm{grad} f(x_k), g_k\rangle+Lt_k^2\langle\mathrm{grad} f(x_k), g_k\rangle-\frac{Lt_k^2(1-\nu^2)}{2} \|\mathrm{grad} f(x_k)\|^2\nonumber\\
    &\le f(x_k)-t_k(1-\nu)\|\mathrm{grad} f(x_k)\|^2+ Lt_k^2(1+\nu)\|\mathrm{grad} f(x_k)\|^2-\frac{Lt_k^2(1-\nu^2)}{2} \|\mathrm{grad} f(x_k)\|^2\nonumber\\
    &=f(x_k)-\frac{t_k}{2}\Big(2-2\nu-Lt_k (1+\nu)^2\Big) \|\mathrm{grad} f(x_k)\|^2\nonumber\\
    &\le f(x_k)-\frac{\delta t_k}{2}\|\mathrm{grad} f(x_k)\|^2\label{descent esti 1}\text{ for all }k\ge N.
\end{align}
It follows from the above that the sequence $\{f(x_k)\}_{k\ge N}$ is nonincreasing, and hence the condition $\inf_{k\in\mathbb{N}} f(x_k)>-\infty$ ensures the convergence of $\{f(x_k)\}$. This allows us to deduce from \eqref{descent esti 1} that
\begin{align}\label{seri tk fxk con}
    \frac{\delta}{2} \sum_{k=N}^\infty t_k\|\mathrm{grad} f(x_k)\|^2\le \sum_{K=N}^\infty (f(x_k)-f(x_{k+1}))\le f(x_k)-\inf_{k\in\mathbb{N}} f(x_k)<\infty.
\end{align}
Combining the latter with \eqref{two inequality} and $x_{k+1}=\mathcal{R}_{x_k}(-t_k g_k)$ gives us 
\begin{align}\label{common lemma}
 \sum_{k=1}^\infty \|x_{k+1}-x_k\| \cdot\|g_k\|  \le \sum_{k=1}^\infty \kappa t_k\|g_k\|^2\le \kappa (1+\nu)^2 \sum_{k=1}^\infty t_k\|\mathrm{grad} f(x_k)\|^2<\infty.
\end{align}

Now we are ready to verify all the assertions of the theorem. Let us start with \ref{thm3-stationary point} and show that $0$ in an accumulation point of $\{g_k\}$. Indeed, supposing the contrary gives us $\epsilon>0$ and $K\in\mathbb{N}$ such that $\|g_k\|\ge \epsilon$ for all $k\ge K$, and therefore 
\begin{align*}
\infty> \sum_{k=K}^\infty t_k\|g_k\|^2\ge \sum_{k=K}^\infty t_k=\infty,
\end{align*}
which is a contradiction justifying that $0$ is an accumulation point of $\{g_k\}$. If $\bar x$ is an accumulation point of $\{x_k\},$ then by Lemma~\ref{stationary point lemma} and \eqref{common lemma}, we find an infinite set $J\subset N$ such that $x_k\overset{J}{\rightarrow}\bar x$ and $g_k\overset{J}{\rightarrow}0.$ The latter being combined with \eqref{two inequality} gives us $\mathrm{grad} f(x_k)\overset{J}{\rightarrow}0$, which yields the stationary condition $\mathrm{grad} f(\bar x)=0.$

To verity \ref{thm3-f(x_k)}, let $\bar x$ be an accumulation point of $\{x_k\}$ and find an infinite set $J\subset \mathbb{N}$ such that $x_k\overset{J}{\rightarrow}\bar x.$ Combining this with the continuity of $f$ and the fact that $\{f(x_k)\}$ is convergent, we arrive at the equalities
\begin{align*}
f(\bar x)= \lim_{k\in J}f(x_k)= \lim_{k\in \mathbb{N}}f(x_k),
\end{align*}
which therefore justify assertion \ref{thm3-f(x_k)}.

To proceed with the proof of \ref{thm3-grad}, assume that $\mathrm{grad} f$ is Lipschitz continuous with constant $L>0$ and employ Lemma~\ref{three sequences lemma} with $\alpha_k:=\|\mathrm{grad} f(x_k)\|$, $\beta_k:=Lt_k(1+\nu)$, and $\gamma_k:=0$ for all $k\in\mathbb{N}$ to derive that $\mathrm{grad} f(x_k)\rightarrow0.$ Observe first that  condition \eqref{a} of this lemma is satisfied due to the the estimates
\begin{align*}
    \alpha_{k+1}-\alpha_k&=\|\mathrm{grad} f(x_{k+1})\|-\|\mathrm{grad} f(x^{k})\|\\
    &\le \|\mathcal{T}_{x_{k+1}}^{x_k}\mathrm{grad} f(x_{k+1})-\mathrm{grad} f(x_k)\|\\
    &=Lt_k\|g_k\|\le Lt_k (1+\nu)\|\mathrm{grad} f(x_k)\|=\beta_k\alpha_k.
\end{align*}
The conditions in \eqref{b} of the lemma are satisfied since $\{t_k\}$ is bounded, $\sum_{k=1}^\infty t_k =\infty$ by \eqref{stepsize Alg-RGDr}, $\gamma_k=0$, and
\begin{align*}
    \sum_{k=1}^\infty \beta_k\alpha_k^2=L(1+\nu)\sum_{k=1}^\infty t_k\|\mathrm{grad} f(x_k)\|^2<\infty,
\end{align*}
where the inequality follows from \eqref{seri tk fxk con}. Thus applying Lemma~\ref{three sequences lemma} gives us $\mathrm{grad} f(x_k)\rightarrow0$ as $k\to\infty$.
\end{proof}

The following theorem on the convergence of iterates to stationary points is ensured under the Riemannian KL property of the objective function, with convergence rates determined by the Riemannian KL exponent.

\begin{theorem}\label{theo Alg-RGDr KL}
    Under the same condition as in Theorem \ref{theo Alg-RGDr} and $\mathcal{S}$ denote the set of all accumulation points.   
   \begin{enumerate}[label=(\roman*)]
    \item\label{x_k conver} 
    Assume the stepsizes are bounded away from $0$, if $f$ satisfies the Riemannian KL property at some accumulation point in $\mathcal{S}$, then $\{x_k\}\rightarrow\bar x$.
   \item\label{conver rate}
   Assume the Riemannian KL property in \ref{x_k conver} holds with the desingularizing function $\varphi(t)=\frac{C}{\theta}t^{\theta}$ with $C>0$ and $\theta \in (0,1)$. Then either $\{x_k\}$ stops finitely at a stationary point, or the following convergence rates are achieved:
    \begin{itemize}
        \item if $\theta=1$, then there exists $k_1$ such that $x_k = \bar x$ for all $k > k_1$.
        \item If $\theta\in[\frac{1}{2},1)$, then there exists $C_r>0$ and $Q\in [0,1)$ such that for all $k$ 
        $$\|x_k-\bar x\| <C_rQ^k.$$
        \item If $\theta\in (0,\frac{1}{2})$, then there exists a positive constant $\title{C_r}>0$ such that for all $k$
        $$\|x_k-\bar x\|<\title{C_r}k^{-\frac{1-\theta}{2\theta-1}}.$$
    \end{itemize}
  \end{enumerate} 
\end{theorem}

\begin{proof}
    To prove \ref{x_k conver}.
    Since $0$ is an accumulation point of $\{g_k\}$, and the application of Lemma \ref{R-dist} and  implies that
\begin{align}\label{Ostrowski condition:2}
    \|x_{k+1}-x_k\| =\|\mathcal{R}_{x_k}(t_kg_k)-x_k \|\le \kappa t_k\|g_k\| \rightarrow 0.
\end{align}
Then by \cite{bolte2014proximal}, we konw that $\mathcal{S}$ is a compect set. Morever, since $f(x_k)$ is nonincreasing, and $f$ is continuous, thus $f$ has the same value at all the points in $\mathcal{S}$. Therefore, by Lemma \ref{RKL:compact}, there exists a single desingularizing function $\varphi$, for the Riemannian KL property of $f$ to hold at all the points in $\mathcal{S}$. Since $f(x_k)\rightarrow f(\bar x)$, $ \inf_{\bar x \in \mathcal{S}}\|x_k-\bar x\|\rightarrow 0$, thus there exists an $l>0$ such that 
\begin{align}\label{R-KL}
    \varphi^\prime(f(x_k)-f(\bar x))\ge \|\mathrm{grad} f(x_k)\|^{-1} \text{ for all }k\ge l.
\end{align}
By Assumption~\ref{Lips def}, we have 
\begin{align}\label{invertible T}
    \|\mathrm{grad} f(\mathcal{R}_{x_k}(-t_kg_k)-\mathcal{T}_{\mathcal{R}_{x_k}(-t_kg_k)}^{-\sharp} \mathrm{grad} f(x_k)\|\le Lt_k\|g_k\| \text{ for all }k\ge l.
\end{align}
Since the vector transport is isometric, it holds that 
\begin{align}\label{isometric}
\|\mathcal{T}_{x_{k+1}}^{-\sharp} \mathrm{grad} f(x_k)\|= \|\mathrm{grad} f(x_k)\|.
\end{align}
Combining the triangle inequality with \eqref{invertible T}, and \eqref{isometric}, we get
\begin{align*}
    \|\mathrm{grad} f(x_{k+1})\|
    &\le \|\mathcal{T}_{x_{k+1}}^{-\sharp} \mathrm{grad} f(x_k)\|+\|\mathrm{grad} f(x_{k+1})-\mathcal{T}_{x_{k+1}}^{-\sharp} \mathrm{grad} f(x_k)\|\\
    &\le \|\mathrm{grad} f(x_k)\|+Lt_k\|g_k\|.
\end{align*}
Suppose in addition that the sequence $\{t_k\}$ is bounded away from 0 (i.e. there is some $\bar t>0$ such that $t_k>\bar t$ for large $k\in \mathbb{N}$). Using the inequality \eqref{two inequality}, we get
\begin{align*}
    \|\mathrm{grad} f(x_{k+1})\|
    &\le \|\mathrm{grad} f(x_k)\|+Lt_k\|g_k\|\le \frac{1}{1-\nu}\|g_k\|+Lt_k\|g_k\|\\
    &\le (\frac{1}{1-\nu}\frac{1}{t_k}+L)t_k\|g_k\|\le (\frac{1}{1-\nu}\frac{1}{\bar t}+L)t_k\|g_k\|\\
    &\le \tilde{L}t_k\|g_k\|,
    %&\le \epsilon+Lt_k\|g_k\|.
\end{align*}
where $\tilde{L}=\frac{1}{(1-\nu)}\frac{1}{\bar t}+L$ is a constant.
Thus, we have
\begin{align}\label{recursion}
    \|\mathrm{grad} f(x_k)\|\le \tilde{L}t_{k-1}\|g_{k-1}\| \text{ for all }k\ge l.
\end{align}
Inserting this into \eqref{R-KL} gives 
\begin{align}\label{R-KL inequality}
    \varphi^\prime(f(x_k)-f(\bar x))\ge \tilde{L}^{-1}(t_{k-1}\|g_{k-1}\|)^{-1} \text{ for all }k\ge l.
\end{align}
Moreover, it follows from \eqref{two inequality} and \eqref{descent esti 1} that 
\begin{align}
    f(x_{k+1})\le f(x_k)-\frac{\delta t_k}{2(1+\nu)^2}\|g_k\|^2,
\end{align}
and the concavity of $\varphi$ yields that
\begin{align}
   \varphi (f(x_k)-f(\bar x))- \varphi (f(x_{k+1})-f(\bar x))
   &\ge \varphi ^\prime(f(x_k)-f(\bar x))(f(x_k)-f(x_{k+1})) \nonumber\\
   &\ge \|\mathrm{grad} f(x_k)\|^{-1} \frac{\delta}{2(1+\nu)^2}t_k\|g_k\|^2 \nonumber\\
   &\ge \tilde{L}^{-1}\frac{\delta}{2(1+\nu)^2}\frac{t_k\|g_k\|^2}{t_{k-1}\|g_{k-1}\|}\label{varphi ineq}
\end{align}
For convenience, we define for all $p,q\in \mathbb{N}$ and $\bar x$ the following quantities 
$\Delta_{p,q}:=\varphi (f(x_p)-f(\bar x))- \varphi (f(x_q)-f(\bar x))$ and $M:=\frac{2(1+\nu)^2 \tilde{L}}{\delta}\in (0,\infty)$.
Consequently, \eqref{varphi ineq} is equivalent to
\begin{align}
    \Delta_{k,k+1}\ge \frac{t_k\|g_k\|^2}{M t_{k-1}\|g_{k-1}\|},
\end{align}
for all $k>l$ and hence
\begin{align*}
    t_k\|g_k\|^2\le M \Delta_{k,k+1} t_{k-1}\|g_{k-1}\|.
\end{align*}
Using the fact that the AM-GM inequality, we infer 
\begin{align}\label{AM-GM}
    2t_k\|g_k\|=2\sqrt{t_k}\sqrt{t_k\|g_k\|^2}\le 2\sqrt{t_{k-1}\|g_{k-1}\|M\Delta_{k,k+1}}\le t_{k-1}\|g_{k-1}\|+M\Delta_{k,k+1}.
\end{align}
Let us now prove that for all $k>l$ the following inquality holds 
\begin{align*}
    \sum_{i=l+1}^k t_k\|g_k\|\le t_l\|g_l\|+M\Delta_{l+1,k+1}.
\end{align*}
Summing up \eqref{AM-GM} for $i=l+1,\cdots,k$ yields
\begin{align*}
    2\sum_{i=l+1}^k t_i\|g_i\|
    &\le \sum_{i=l+1}^k t_{i-1}\|g_{i-1}\|+M\sum_{i=l+1}^k\Delta_{i,i+1}\\
    &\le \sum_{i=l+1}^k t_i\|g_i\|+t_l\|g_l\|+M\sum_{i=l+1}^k\Delta_{i,i+1}\\
    &= \sum_{i=l+1}^k t_i\|g_i\|+t_l\|g_l\|+M\Delta_{l+1,k+1},
\end{align*}
where the last inequality follows from the fact that $\Delta_{p,q}+\Delta_{q,r}=\Delta_{p,r}$ for all $p,q,r\in \mathbb{N}$. Since $\varphi>0$, we thus have for all $k>l$ that
\begin{align*}
    \sum_{i=l+1}^k t_i\|g_i\|\le t_l\|g_l\|+M\varphi(f(x_{l+1})-f(\bar x)).
\end{align*}
Combining Assumption \ref{R-dist} and \eqref{Ostrowski condition:2}, this easily shows that the sequence $\{x_k\}_{k\in \mathbb{N}}$ has finite length, that is,
\begin{align}
    \sum_{k=1}^{\infty} \|x_{k+1}-x_k\| \le \kappa \sum_{k=1}^{\infty} t_k\|g_k\|\le \infty.
\end{align}
The latter means that $\{x_k\}$ is a Cauchy sequence, and tell us that $\{x_k\}$ is converges to $\bar x$.

Let us now verify assertion \ref{conver rate} of the theorem. Applying $\varphi(t)=\frac{C}{\theta}t^{\theta}$ to \eqref{varphi ineq} yields
\begin{align}\label{apply varphi}
    t_k\|g_k\|^2\le t_{k-1}\|g_{k-1}\| \frac{2(1+\nu)^2\tilde{L}C}{\delta \theta}\Big((f(x_k)-f(\bar x))^{\theta}- (f(x_{k+1})-f(\bar x))^{\theta}\Big),
\end{align}
for all $k\ge l$. Taking square root to the both sides of \eqref{apply varphi} and using the AM-GM inequality, we have 
\begin{align}
    2t_k\|g_k\|=t_{k-1}\|g_{k-1}\|+\frac{2(1+\nu)^2\tilde{L}C}{\delta \theta}\Big((f(x_k)-f(\bar x))^{\theta}- (f(x_{k+1})-f(\bar x))^{\theta}\Big),
\end{align}
for all $k\ge l$. Summing the both sides from $p>l$ to $\infty$ yields
\begin{align}\label{sum-1}
    \sum_{k=p}^{\infty} t_k\|g_k\|\le t_{p-1}\|g_{p-1}\|+\frac{2(1+\nu)^2\tilde{L}C}{\delta \theta}(f(x_p)-f(\bar x))^{\theta}.
\end{align}
By \eqref{R-KL}, we have $\frac{1}{C}(f(x_k)-f(\bar x))^{1-\theta}\le \|\mathrm{grad}f(x_k)\|$. Combining this inequality with \eqref{recursion} yields
\begin{align}\label{theta equ}
    \frac{1}{C}(f(x_k)-f(\bar x))^{1-\theta}\le \tilde{L}t_{k-1}\|g_{k-1}\|.
\end{align}
It follows from \eqref{sum-1} and \eqref{theta equ} that 
\begin{align}\label{sum-2}
    \sum_{k=p}^{\infty} t_k\|g_k\|\le t_{k-1}\|g_{k-1}\|+\frac{2(1+\nu)^2\tilde{L}C}{\delta \theta}(CLt_{k-1}\|g_{k-1}\|)^{\frac{\theta}{1-\theta}}, \text{ for all }k\ge l.
\end{align}
Define $\Delta_k=\sum_{i=k}^{\infty}\|g_i\|$. Therefore, inequality \eqref{sum-2} becomes
\begin{align}\label{sum-3}
    \Delta_k\le (\Delta_{k-1}-\Delta_k)+b_1(\Delta_{k-1}-\Delta_k)^{\frac{\theta}{1-\theta}},\text{ for all }k\ge l,
\end{align}
where $b_1=\frac{2(1+\nu)^2\tilde{L}C}{\delta \theta}(C\tilde{L})^{\frac{\theta}{1-\theta}}$. Noting that \eqref{sum-3} has the same form as \cite{attouch2009convergence}, we have follow the same derivations in \cite{attouch2009convergence} and show that  
\begin{itemize}
    \item if $\theta=1$, then Algorithm~\ref{Alg-RGDr} terminates in finite steps,\\
    \item if $\theta\in [\frac{1}{2},1)$, then $\Delta_k<C_rQ^k$ for $C_r>0$ and $Q\in [0,1)$,\\
    \item if $\theta\in (0,\frac{1}{2})$, then $\Delta_k<\title{C_r}k^{-\frac{1-\theta}{2\theta-1}}$ for $\title{C_r}>0$.
\end{itemize}
It only remains to show that $\mathrm{dist}(x_k,\bar x)<C_p \Delta_k$ for a positive constant $C_p$. This can be obtained by  
\begin{align*}
    \|x_k-\bar x\|\le \sum_{i=k}^{\infty} \|x_{i+1}-x_i \|\le \kappa \sum_{i=k}^{\infty}t_i\|g_i\|=\kappa \Delta_k,
\end{align*}
where the first inequality is by triangle inequality and the second inequality is from Assumption \ref{assu desi}.
This completes the proof.
\end{proof}

\section{Applications}\label{sec app}
In this section, we apply the inexact Riemannian gradient method for $\mathcal{C}^1$-smooth nonconvex optimization, as developed in Sect.~\ref{sec inexact RGD}, to design and justify two novel gradient-based inexact methods. 

\subsection{Riemannian sharpness-aware minimization}\label{app:rsam}
Sharpness-aware minimization (SAM) is a recently proposed gradient-based optimization method that significantly enhances the generalization of deep neural networks \cite{foret2020sharpness}. While the fundamental convergence properties of SAM have been established in \cite{khanh2024fundamental}, these results are limited to Euclidean space. In this subsection, we introduce a novel construction of Sharpness-Aware Minimization on Riemannian manifolds and present the convergence analysis of RSAM from an inexact gradient perspective.

\begin{algorithm}
\caption{Riemannian Sharpness-Aware Minimization (RSAM)}\label{RSAM}
\begin{algorithmic}[1]
\STATE Choose some initial point $x_1\in\mathcal{M},$ sequence of perturbation radii $\{\rho_k\}\subset (0,\infty)$, and sequence of stepsizes $\{t_k\}\subset(0,\infty).$ For $k=1,2,\ldots,$ do the following
\STATE Set $x_{k+1}=\mathcal{R}_{x_k}(-t_k \mathcal{T}_{x_k^{adv}}^{x_k} \mathrm{grad} f(x_k^{adv})),$ where $x_k^{adv}=\mathcal{R}_{x_k}(\rho_k \frac{\mathrm{grad} f(x)}{\|\mathrm{grad} f(x)\|})$.   
\end{algorithmic}
\end{algorithm}

In fact, we show that Algorithm~\ref{RSAM} is a special case of IRGD framework presented in Algorithm~\ref{Alg-RGD}. As a result, it inherits all the convergence properties established in Theorem \ref{convergence Alg-RGD} and \ref{convergence iterates}. The following theorem verifies that Algorithm~\ref{RSAM} satisfies the inexact unnormalized gradient condition given in \eqref{unnormalized-r}.

\begin{theorem} Suppose that Assumption \ref{descent condition},\ref{lower bounded}, \ref{Lips def} and \ref{assu desi} holds. Let $f:\mathcal{M}\rightarrow\mathbb{R}$ be a $\mathcal{C}^{1,L}$ function, and let $\{x_k\}$ be generated by Algorithm~\ref{RSAM} with the parameters
\begin{align}\label{parameter SAM + VASSO}
 \sum_{k=1}^\infty t_k=\infty,\;t_k\downarrow0, \sum_{k=1}^\infty t_k \rho_k<\infty,\;\limsup \rho_k<\frac{2}{L}.    
\end{align}
Then all the convergence properties presented in Theorem~\ref{convergence Alg-RGD} hold.
\end{theorem}

\begin{proof}
Considering Algorithm~\ref{RSAM} and defining $$g_k = \mathcal{T}_{x_k^{adv}}^{x_k} \mathrm{grad} f(\mathcal{R}_{x_k}(\rho_k \frac{\mathrm{grad} f(x)}{\|\mathrm{grad} f(x)\|})),$$ we deduce that
\begin{align*}
    \|g_k-\mathrm{grad} f(x_k)\|
    &\le L\|\mathcal{R}_{x_k}(\rho_k \frac{\mathrm{grad} f(x)}{\|\mathrm{grad} f(x)\|})-x_k\|\\
    &=L\|\rho_k\frac{\mathrm{grad} f(x)}{\|\mathrm{grad} f(x)\|}\| \le L\rho_k.
\end{align*}
Therefore, Algorithm~\ref{RSAM} is a specialization of Algorithm~\ref{Alg-RGD} with $\epsilon_k = L\rho_k.$ Combining this with \eqref{parameter SAM + VASSO} also gives us \eqref{parameter general}, thereby verifying all the assumptions and conditions in Theorem \ref{convergence Alg-RGD} and \ref{convergence iterates}. Consequently, all the convergence properties outlined in Theorem \ref{convergence Alg-RGD} and \ref{convergence iterates} hold for Algorithm~\ref{RSAM}.
\end{proof}

This convergence analysis ensures the last-iterate convergence of RSAM, offering a stronger conclusion than the average-iterate convergence discussed in \cite{yun2024riemannian}, despite the complexity guarantee provided therein.

\subsection{Riemannian extragradient}\label{app:reg}
Riemannian extragradient (REG) was initially proposed as a novel first-order method designed for monotone Riemannian variational inequality problems, with both last-iterate and average-iterate convergence established under strong assumptions \cite{hu2024extragradient}. In contrast, in this subsection, we propose a novel REG method focused on analyzing nonconvex smooth problems from an inexact gradient perspective, offering convergence analysis under more general assumptions.

\begin{algorithm}
\caption{Riemannian Extragradient} \label{REG}
\begin{algorithmic}[1]
\STATE
        Choose some initial point $x_1\in\mathcal{M},$ a sequence of perturbation radii $\{\rho_k\}\subset [0,\infty)$, and a sequence of stepsizes $\{t_k\}\subset[0,\infty).$ For $k=1,2,\ldots,$ do the following
\STATE
        Set $x_{k+1}=\mathcal{R}_{x_k}(-t_k \mathcal{T}_{x_k^{adv}}^{x_k} \mathrm{grad} f (x_k^{adv}))$,
        where $x_k^{adv}=\mathcal{R}_{x_k}(-\rho_k \mathrm{grad} f(x_k)).$
\end{algorithmic}
\end{algorithm}

It is demonstrated that Algorithm~\ref{REG} is a specific instance of IRGDr utilizing relative errors, as outlined in Algorithm~\ref{Alg-RGDr}. Consequently, it inherits all the convergence properties detailed in Theorem \ref{theo Alg-RGDr} and \ref{theo Alg-RGDr KL}. The following theorem verifies that Algorithm~\ref{REG} satisfies the inexact normalized gradient condition given in \eqref{normalized-r}.

\begin{theorem}
Suppose that Assumption \ref{descent condition},\ref{lower bounded}, \ref{Lips def} and \ref{assu desi} holds. Let $f:\mathcal{M}\rightarrow\mathbb{R}$ be a ${\mathcal{C}}^1$-smooth function satisfying the descent condition with some constant $L>0.$ Let $\{x_k\}$ be the sequence generated by Algorithm~\ref{REG} with $\rho_k\le \frac{\nu}{L}$ for some $\nu \in[0,1)$ and with the stepsize satisfying \eqref{stepsize Alg-RGDr}. Then all the convergence properties in Theorem \ref{theo Alg-RGDr} and \ref{theo Alg-RGDr KL} hold. 
\end{theorem}
\begin{proof} Defining $g_k:=\mathcal{T}_{x_k^{adv}}^{x_k} \mathrm{grad} f(\mathcal{R}_{x_k}(-\rho_k \mathrm{grad} f(x_k)))$ and utilizing  $\rho_k \le \frac{\nu}{L}$, we obtain
\begin{align*}
    \|g_k-\mathrm{grad} f(x_k)\|
    &=\|\mathcal{T}_{x_k^{adv}}^{x_k} \mathrm{grad} f(\mathcal{R}_{x_k}(-\rho_k \mathrm{grad} f(x_k)))-\mathrm{grad} f(x_k)\|\\
    &\le L\|-\rho_k \mathrm{grad} f(x_k)\|\le \nu \|\mathrm{grad} f(x_k)\|,
\end{align*}
which verifies the inexact condition in Step~2 of Algorithm~\ref{Alg-RGDr}. Therefore, all the convergence properties in Theorem \ref{theo Alg-RGDr} and \ref{theo Alg-RGDr KL} hold for Algorithm~\ref{REG}. 
\end{proof}

In \cite{hu2024extragradient}, the REG method was proposed without any numerical experimental verification. In this paper, I validate the efficiency of this method on the PCA problem in Sect.~\ref{pca}.

\section{Numerical Experiments}\label{sec numerical}
In this section, we present numerical experiments demonstrating the efficiency of our proposed IRGD method, implemented by Algorithm \ref{Alg-RGDr}, in comparison with the RGD method \eqref{eq:rgd} for solving the low-rank matrix completion (MC) problem \cite{vandereycken2013low,sakai2021sufficient}. Additionally, we compare the performance of our REG method, proposed by Algorithm \ref{REG}, with the RGD method \eqref{eq:rgd} in the context of principal component analysis (PCA) \cite{li2023stochastic,zhou2019faster}.  All numerical experiments were implemented in Python 3.7.6 on a laptop with Intel(R) Xeon(R) Gold 6230 CPU @ 2.10GHz.

In all of our tests, the initial iterative point is randomly generated. We set the diminishing stepsize $t_k=\frac{1}{k^{\alpha}}$, with $\alpha \in (0,1)$, which satisfies the conditions in \eqref{diminishing}. Additionally, we perform a line search using the Pymanopt 0.2.6rc1 package \cite{townsend2016pymanopt} without any modifications. For the MC problem, we chose $\alpha = 0.1$, while for the PCA problem, $\alpha = 0.75$. To generate the inexactness for testing purposes, we create an inexact gradient $g_k$ by adding a random vector with the norm $\nu \delta_k$ to the exact gradient $\mathrm{grad} f(x_k)$. To ensure that the controlled errors between the exact and inexact gradients do not decrease too quickly, we select $\delta_k=(k+1)^{-p}$ with $ p\ge 1$, where we set $p=2.1$. If the stopping condition
$$\|\mathrm{grad}f(X_k)\|_{X_k}<10^{-6}$$
was satisfied, we determined that a sequence had converged to an optimal solution. Detailed information on the numerical experiments and the achieved numerical results is presented in Tables \ref{tab_mc}, \ref{tab_pca} and \ref{tab_mnist}. In these tables, the terms 'iter,', 'time,' denote the total number of the iterations and the CPU time, respectively, with 10000 being the maximum number of iterations.

\subsection{Low-rank matrix completion}\label{mc}
Low-rank matrix completion (MC) aims to recover a low-rank matrix $X \in \mathbb{R}^{m\times n}$ from a subset $\Omega$ of the complete set of entries $\{1,\cdots,m\}\times \{1,\cdots,n\}$. For $A\in \mathbb{R}^{m\times n}$, 
$$ \min_{X\in M_k} \|P_{\Omega}(X-A)\|_F^2,$$
where $M_k:=\{X\in \mathbb{R}^{m\times n}: \mathrm{rank}(X)=k\}$ is a FixedRankEmbedded manifold, $\|\cdot\|_F$ denotes the Frobenius norm and $ P_{\Omega}$ is a linear operator that extracts entries in $\Omega$ and fills the entries not in $\Omega$ with zeros.
In our experiment, we set  different dimension $m\times n$, and $\Omega$ contained each pair $(i,j) \in \{1,\cdots,m\}\times \{1,\cdots,n\}$ with probability $\frac{1}{2}$. Moreover, $A$ is a randomly generated data matrix. 

% \begin{table}
% % table caption is above the table
% \caption{Numerical results of MC problem with different $m,n,r$.}
% \label{tab_mc}       % Give a unique label
% % For LaTeX tables use
% \begin{tabularx}{\textwidth}{lXccXccXcc}
% \hline\noalign{\smallskip}
% Solver & $\nu$ & \multicolumn{2}{c}{$m=n=20, r=8$} & \multicolumn{2}{c}{$m=n=50, r=20$} & \multicolumn{2}{c}{$m=n=100, r=40$} \\
%  &  & Iter & Time & Iter & Time & Iter & Time\\
% \noalign{\smallskip}\hline\noalign{\smallskip}
% RGD-diminishing & - & 1133 & 1.4849 & 641  & 1.2216 & 573 & 18.0899 \\
% IRGD-diminishing & $10^{-8}$ & 1097 & 1.5469 & 633 & 1.2358 & 562 & 19.5612 \\
%  & $10^{-5}$ & 966 & 1.3282 & 690 & 1.3453 & 557 & 19.9368 \\
%  & $10^{-3}$ & 952 & 1.3264 & 625 & 1.2311 & 574 & 20.6600 \\
%  & $10^{-2}$ & 1191 & 1.6910 & 647 & 1.2683 & 589 & 22.7239 \\
% \noalign{\smallskip}\hline\noalign{\smallskip}
% RGD-line search & - & 528 & 0.8945 & 258  & 0.6841 & 236 & 11.4349 \\
% IRGD-line search & $10^{-8}$ & 423 & 0.7656 & 247 & 0.6835 & 244 & 13.2515 \\
%  & $10^{-5}$ & 466 & 0.8022 & 250 & 0.6880 & 236 & 11.8484 \\
%  & $10^{-3}$ & 316 & 0.5403 & 257 & 0.6956 & 254 & 10.9358 \\
%  & $10^{-2}$ & 495 & 0.8650 & 365 & 0.6956 & 453 & 17.4021 \\
% \noalign{\smallskip}\hline
% \end{tabularx}
% \end{table}

\begin{table}[h]
\centering
\setlength{\tabcolsep}{5pt}
\caption{Numerical results of MC problem with different $m,n,r$.}\label{tab_mc}
\begin{tabular}{|c|c|cc|cc|cc|}
\cline{1-8}
\multirow{2}{*}{Solver} & \multirow{2}{*}{$\rho$} & \multicolumn{2}{c}{$m=n=20, r=8$} & \multicolumn{2}{c}{$m=n=50, r=20$} & \multicolumn{2}{c}{$m=n=100, r=40$} \\ \cline{3-8} 
 &  & Iter & Time & Iter & Time & Iter & Time\\ \cline{1-8} 
 RGD-diminishing & - & 1133 & 1.4849 & 641  & 1.2216 & 573 & 18.0899  \\\cline{1-2}
\multirow{4}{*}{REG-diminishing}
 & $10^{-8}$ & 1097 & 1.5469 & 633 & 1.2358 & 562 & 19.5612 \\
 & $10^{-5}$ & 966 & 1.3282 & 690 & 1.3453 & 557 & 19.9368 \\
 & $10^{-3}$ & 952 & 1.3264 & 625 & 1.2311 & 574 & 20.6600 \\
 & $10^{-2}$ & 1191 & 1.6910 & 647 & 1.2683 & 589 & 22.7239 \\ \cline{1-8}
 RGD-line search & - & 528 & 0.8945 & 258  & 0.6841 & 236 & 11.4349\\ \cline{1-2}
\multirow{4}{*}{REG-line search}
 & $10^{-8}$ & 423 & 0.7656 & 247 & 0.6835 & 244 & 13.2515 \\
 & $10^{-5}$ & 466 & 0.8022 & 250 & 0.6880 & 236 & 11.8484 \\
 & $10^{-3}$ & 316 & 0.5403 & 257 & 0.6956 & 254 & 10.9358 \\
 & $10^{-2}$ & 495 & 0.8650 & 365 & 0.6956 & 453 & 17.4021 \\ \cline{1-8}
\end{tabular}
\end{table}

\begin{figure}
\centering  %图片全局居中
\subfigure[]{
\includegraphics[width=4.3cm]{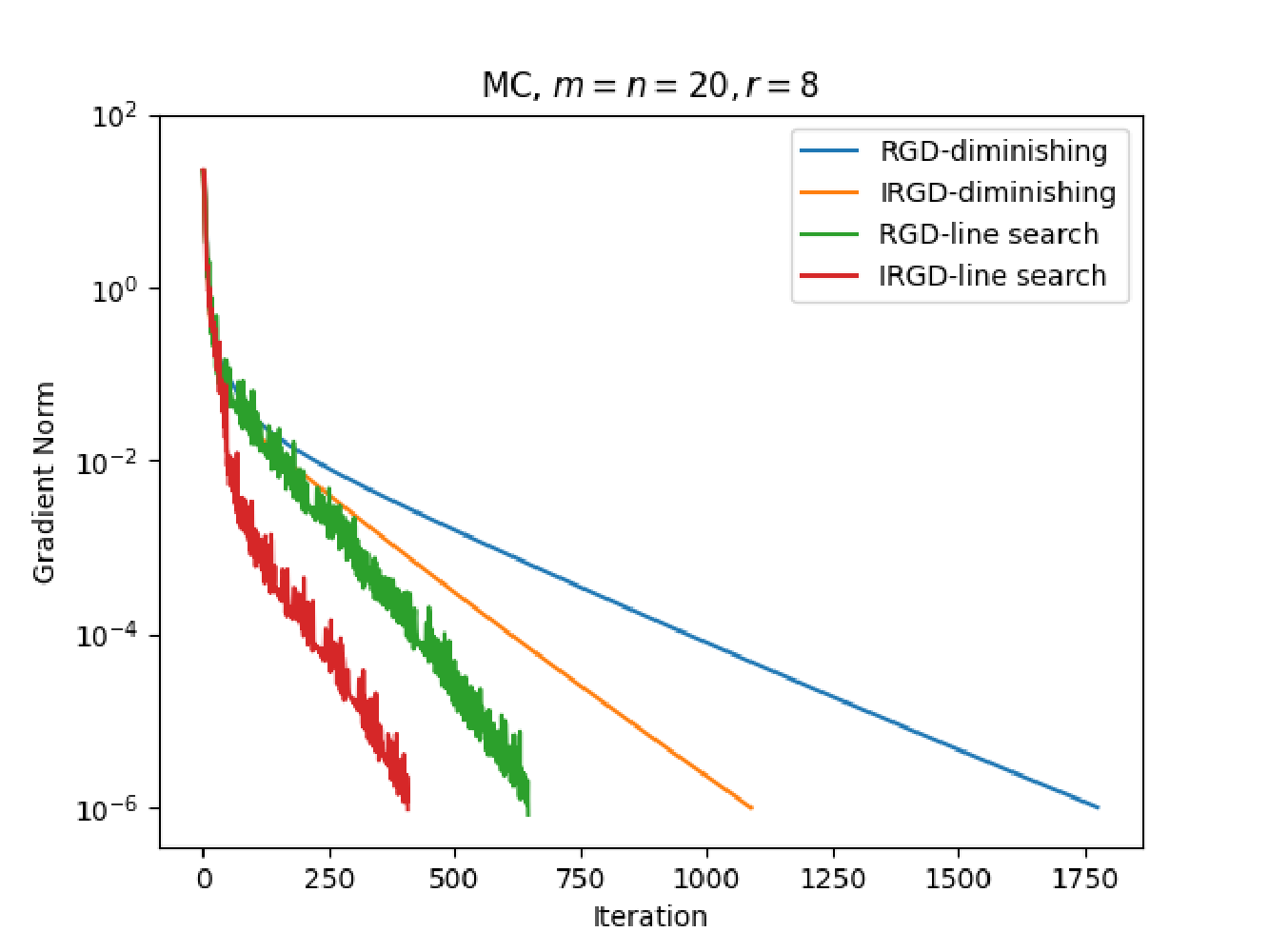}}
\subfigure[]{
\includegraphics[width=4.3cm]{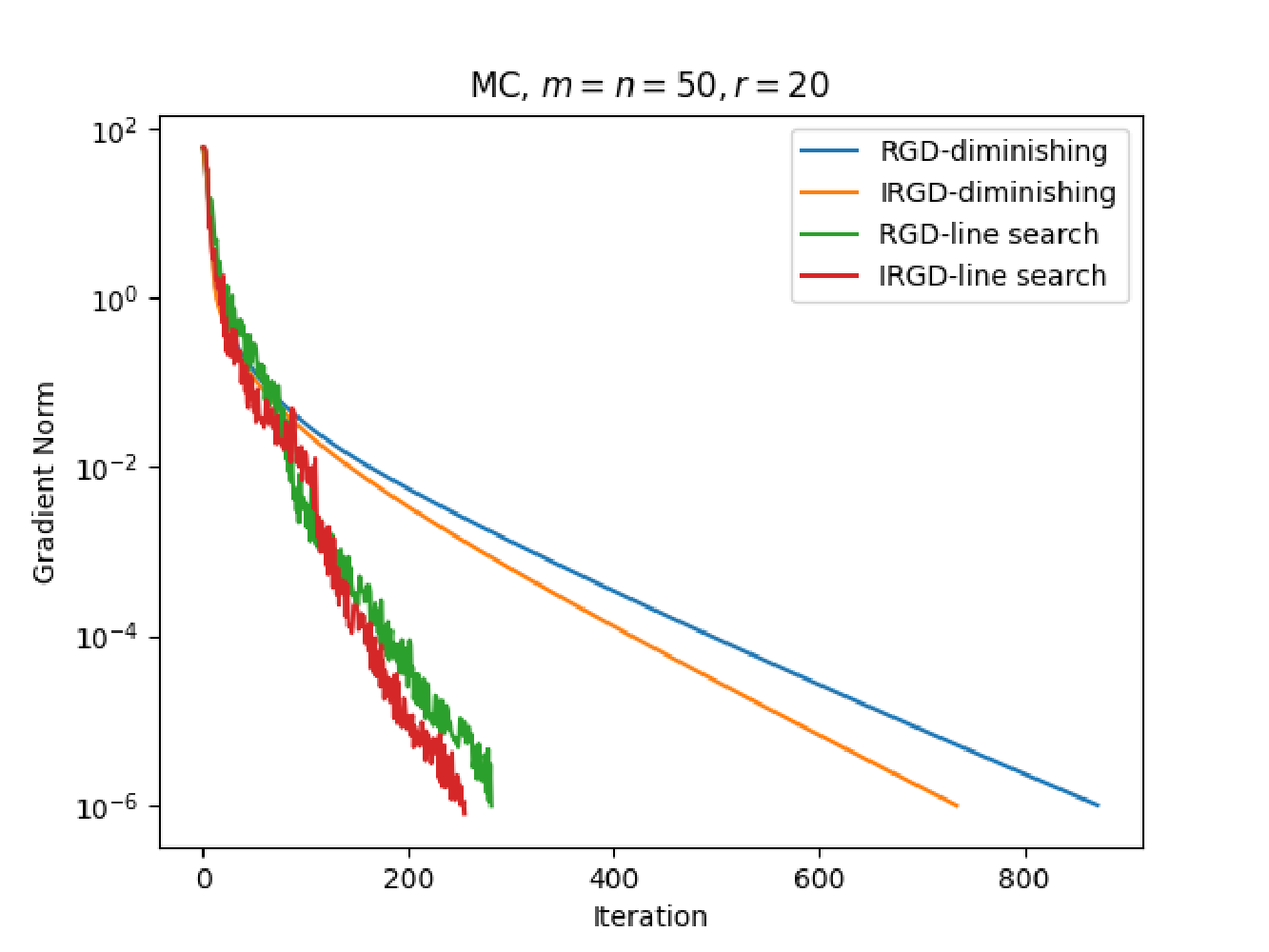}}
\subfigure[]{
\includegraphics[width=4.3cm]{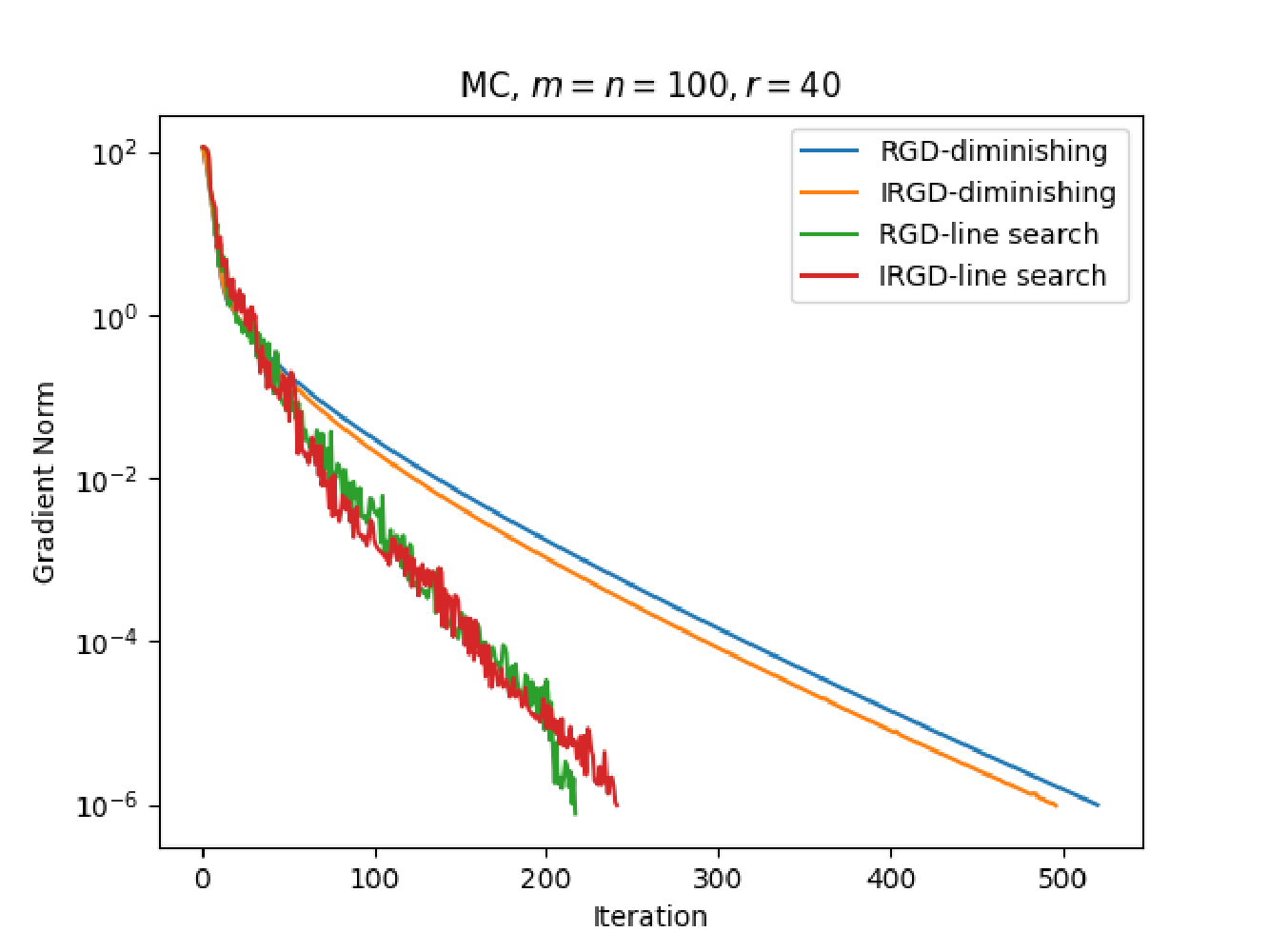}}
\caption{The gradient norm for RGD and IRGD ($\nu=10^{-3}$) on the MC problem.}
\label{fig_mc}
\end{figure}

In Table \ref{tab_mc}, we compare the numerical performance of RGD and IRGD with different selections of the error factor $\nu$. It is evident that the performance of IRGD is slightly better than that of RGD, and the line search-based algorithm generally outperforms the one with a decreasing stepsize. This is consistent with the expectation that RGD can sometimes perform better since it utilizes the exact gradient. Notably, IRGD with $\nu = 10^{-3}$ outperforms RGD in lower dimensions, but shows similar performance to RGD in higher dimensions $m=n=100, r=40$, as depicted in Fig \ref{fig_mc}. These results suggest that IRGD is competitive with RGD, even though the former relies only on inexact gradient information.

\subsection{Principal component analysis}\label{pca}
Principal component analysis (PCA) can be viewed as a Rayleigh-quotient minimization problem on the Grassmann manifold. For $H\in \mathcal{S}_{++}^n$,
$$ \min_{X\in Gr(n,p)} -\frac{1}{2}Tr(X^{\top}HX), $$
where the Grassmann manifold $ Gr(n,p)$ is the set of $p$-dimensional subspaces in $\mathbb{R}^n$ and $\mathcal{S}_{++}^n$ denotes the set of all $n\times n$ symmetric positive definite matrices.
In our experiment, we set different dimension $n\times p$, and the matrix $H$ is generated by $H=AA^{\top}$, where $A\in \mathbb{R}^{n\times p}$ is a randomly generated data matrix. For each run, we sample $m = n\times p$ Gaussian samples for each iteration.

\begin{table}[h]
\centering
\setlength{\tabcolsep}{8pt}
\caption{Numerical results of PCA problem with different $n,p$.}\label{tab_pca}
\begin{tabular}{|c|c|cc|cc|cc|}
\cline{1-8}
\multirow{2}{*}{Solver} & \multirow{2}{*}{$\rho$} & \multicolumn{2}{c|}{$n=20, p=10$} & \multicolumn{2}{c|}{$n=100, p=50$} & \multicolumn{2}{c|}{$n=200, p=100$}
\\ \cline{3-8} 
 &  & Iter & Time & Iter & Time & Iter & Time\\ \cline{1-8} 
 RGD-diminishing & - & 176 & 0.0976 & 1292 & 2.9957 & 3128 & 32.2487 \\\cline{1-2}
\multirow{3}{*}{REG-diminishing}
 & $10^{-8}$ & 191 & 0.1926 & 1292 & 5.6258 & 3205 & 64.2076\\  
 & $10^{-5}$ & 183 & 0.1872 & 1286 & 5.7957 & 3097 & 62.7726\\
 & $10^{-3}$ & 147 & 0.1525 & 744 & 3.5089 & 1698 & 35.1246\\ \cline{1-8}
 RGD-line search & - & 115 & 0.0742 & 749 & 2.1970 & 1214 & 17.0786 \\ \cline{1-2}
\multirow{3}{*}{REG-line search}
 & $10^{-8}$ & 132 & 0.1486 & 825 & 4.1860 & 1132 & 27.0703\\
 & $10^{-5}$ & 125 & 0.1404 & 795 & 3.9235 & 1115 & 27.0823\\
 & $10^{-3}$ & 113 & 0.1283 & 561 & 2.8655 & 402 & 10.8131\\ \cline{1-8}
% success-1e-2 & 68 &  90.7\% & 50 & 66.7\% & 67 & 88.3\%  \\
% success-1e-4 & 64 &  85.3\% & 48 & 64.0\% & 62 & 82.7\%  \\
% fastest & 45 & 60.0\% & 2 & 2.7\% &28 &  37.3\%\\
% % accuracy & \textbf{30} & \textbf{40\%} & 17 &22.7\% &28 & 37.3\% \\
% % fastest under success & 45 &  66.2\% & 1 & 2.0\% & 24 &  36.9\% \\
% % not slower 3 times & 74 &  98.7\% & 21 & 28.0\% & 52 & 69.3\% \\
% % not slower 3 times under success & 52 &  76.5\% & 20 & 40.0\% & 48 & 73.8\%  \\
% geomean time & \multicolumn{2}{c|}{\textbf{137.8}} & \multicolumn{2}{c|}{619.2} & \multicolumn{2}{c|}{233.0}  \\
% \cline{1-7}
\end{tabular}
\end{table}

\begin{figure}
\centering  %图片全局居中
\subfigure[]{
\includegraphics[width=4.3cm]{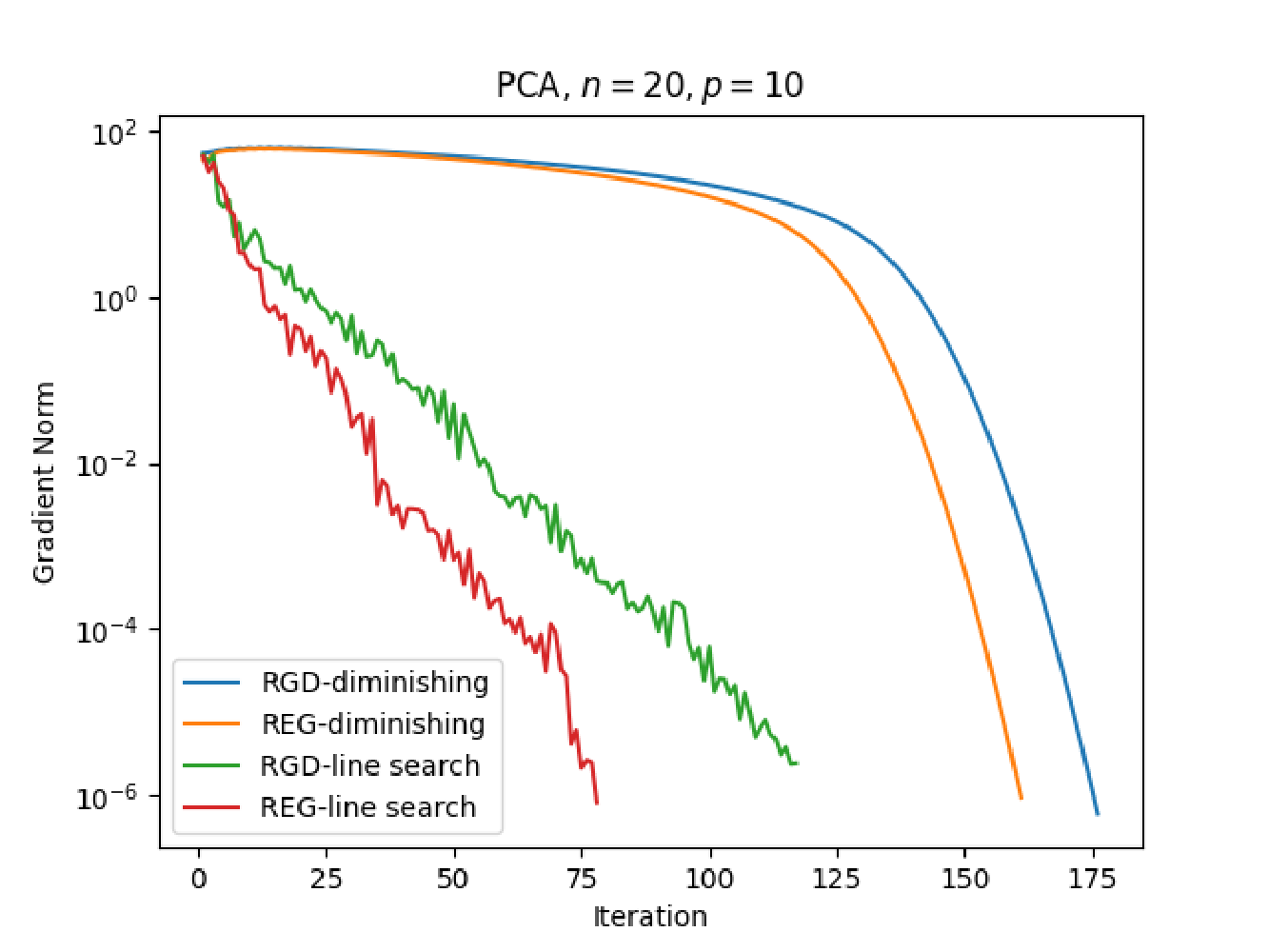}}
\subfigure[]{
\includegraphics[width=4.3cm]{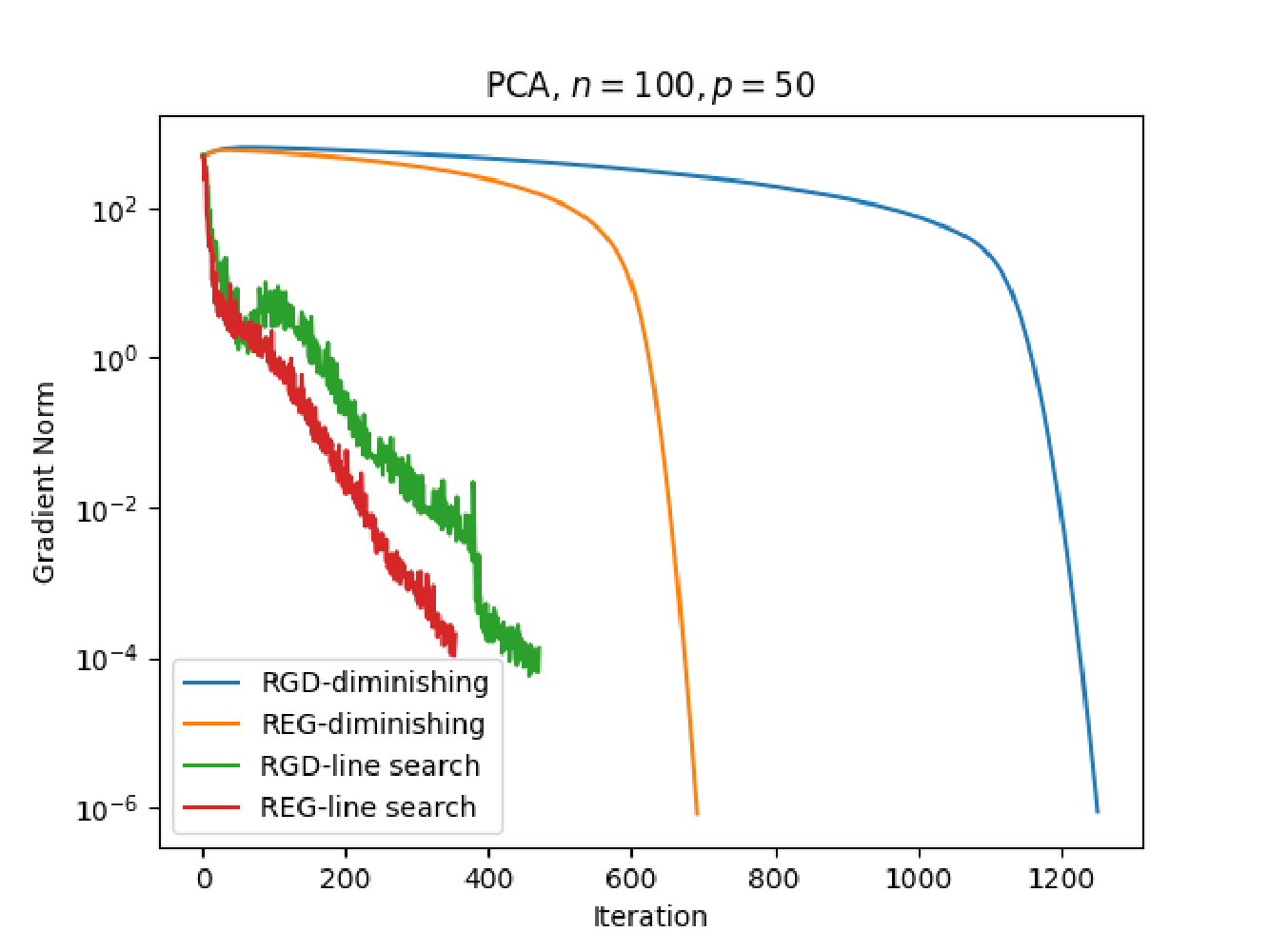}}
\subfigure[]{
\includegraphics[width=4.3cm]{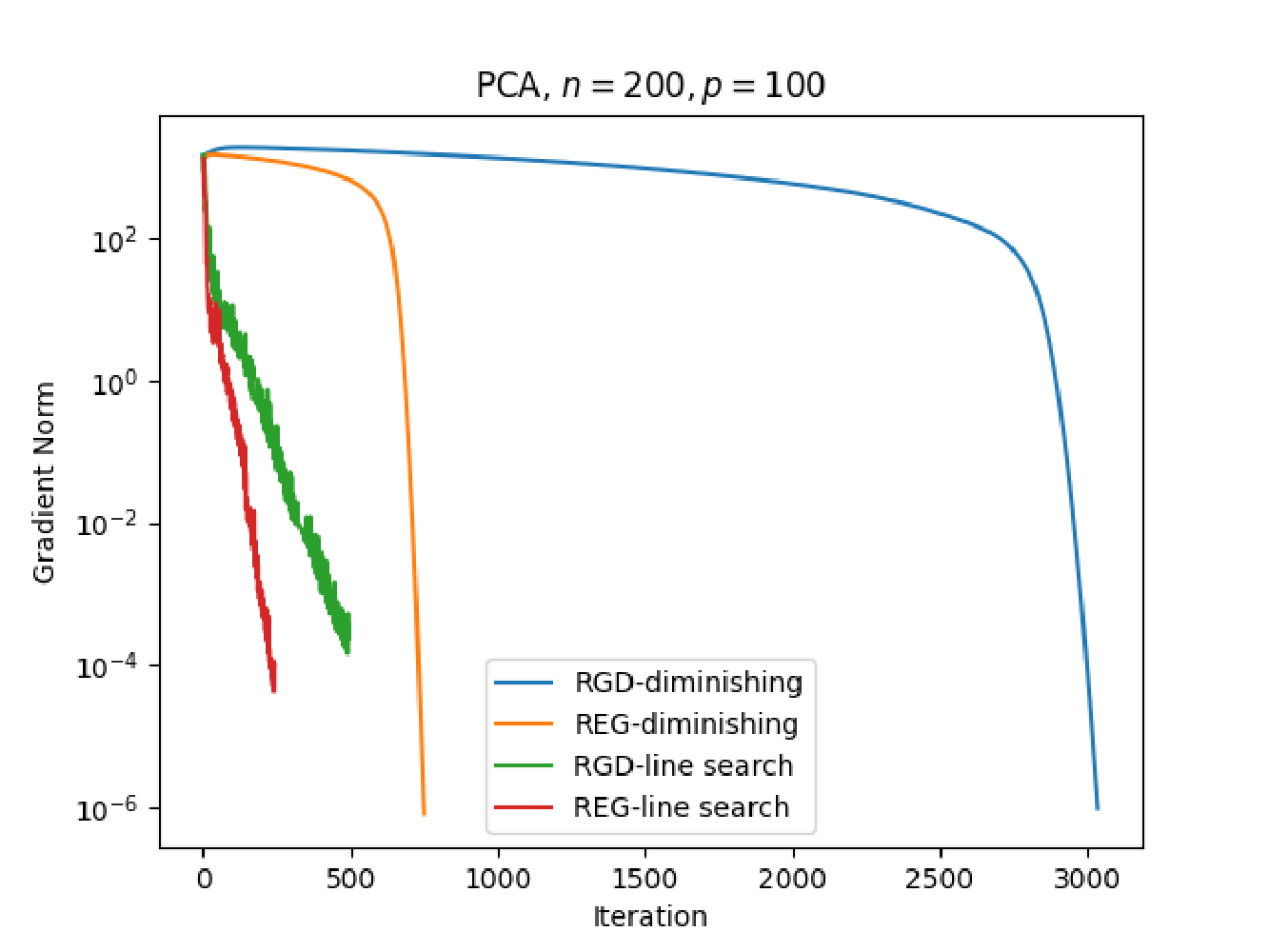}}
\caption{The gradient norm for RGD and REG ($\rho=10^{-3}$) on PCA problem.}
\label{fig_pca}
\end{figure}

% \begin{table}
% % table caption is above the table
% \caption{Numerical results of PCA-MNIST problem with different $n,p$.}
% \label{tab_mnist}       % Give a unique label
% \begin{tabularx}{\textwidth}{lXccXccXcc}  % 使用 tabularx 并设置宽度为 \textwidth
% \hline\noalign{\smallskip}
% Solver & $\rho$ & \multicolumn{2}{c}{$p=2$} & \multicolumn{2}{c}{$p=3$} & \multicolumn{2}{c}{$p=5$}\\ 
%  &  & Iter & Time & Iter & Time & Iter & Time\\
% \noalign{\smallskip}\hline\noalign{\smallskip}
% RGD-diminishing & - & 136 & 0.2602 & 90 & 0.1985 & 95 & 0.2739 \\
% REG-diminishing & $10^{-8}$ & 147 & 0.5311 & 90 & 0.3812 & 96 & 0.5373\\
%  & $10^{-5}$ & 135 & 0.4924 & 90 & 0.3810 & 110 & 0.6041\\
%  & $10^{-3}$ & 143 & 0.5195 & 87 & 0.3559 & 90 & 0.5111\\
%  & $10^{-2}$ & 126 & 0.4632 & 53 & 0.2134 & 72 & 0.3939\\
% \noalign{\smallskip}\hline\noalign{\smallskip}
% RGD-line search & - & 88 & 0.2377 & 49 & 0.1443 & 91 & 0.3514 \\
% REG-line search & $10^{-8}$ & 92 & 0.3876 & 50 & 0.2391 & 88 & 0.5495\\
%  & $10^{-5}$ & 96 & 0.3932 & 48 & 0.2313 & 84 & 0.5256\\
%  & $10^{-3}$ & 93 & 0.3802 & 45 & 0.2211 & 88 & 0.5545\\
%  & $10^{-2}$ & 62 & 0.2893 & 32 & 0.1647 & 66 & 0.4193\\
% \noalign{\smallskip}\hline
% \end{tabularx}
% \end{table}

\begin{table}[h]
\centering
\setlength{\tabcolsep}{8pt}
\caption{Numerical results of PCA problem with different $p$.}\label{tab_mnist}
\begin{tabular}{|c|c|cc|cc|cc|}
\cline{1-8}
\multirow{2}{*}{Solver} & \multirow{2}{*}{$\rho$} & \multicolumn{2}{c|}{$p=2$} & \multicolumn{2}{c|}{$p=3$} & \multicolumn{2}{c|}{$p=5$}\\ \cline{3-8} 
 &  & Iter & Time & Iter & Time & Iter & Time\\ \cline{1-8} 
 RGD-diminishing & - & 136 & 0.2602 & 90 & 0.1985 & 95 & 0.2739 \\\cline{1-2}
\multirow{4}{*}{REG-diminishing}
 & $10^{-8}$ & 147 & 0.5311 & 90 & 0.3812 & 96 & 0.5373\\
 & $10^{-5}$ & 135 & 0.4924 & 90 & 0.3810 & 110 & 0.6041\\
 & $10^{-3}$ & 143 & 0.5195 & 87 & 0.3559 & 90 & 0.5111\\
 & $10^{-2}$ & 126 & 0.4632 & 53 & 0.2134 & 72 & 0.3939\\ \cline{1-8}
 RGD-line search & - & 88 & 0.2377 & 49 & 0.1443 & 91 & 0.3514 \\ \cline{1-2}
\multirow{4}{*}{REG-line search}
 & $10^{-8}$ & 92 & 0.3876 & 50 & 0.2391 & 88 & 0.5495\\
 & $10^{-5}$ & 96 & 0.3932 & 48 & 0.2313 & 84 & 0.5256\\
 & $10^{-3}$ & 93 & 0.3802 & 45 & 0.2211 & 88 & 0.5545\\
 & $10^{-2}$ & 62 & 0.2893 & 32 & 0.1647 & 66 & 0.4193\\ \cline{1-8}
\end{tabular}
\end{table}

\begin{figure}
\centering  %图片全局居中
\subfigure[]{
\includegraphics[width=4.3cm]{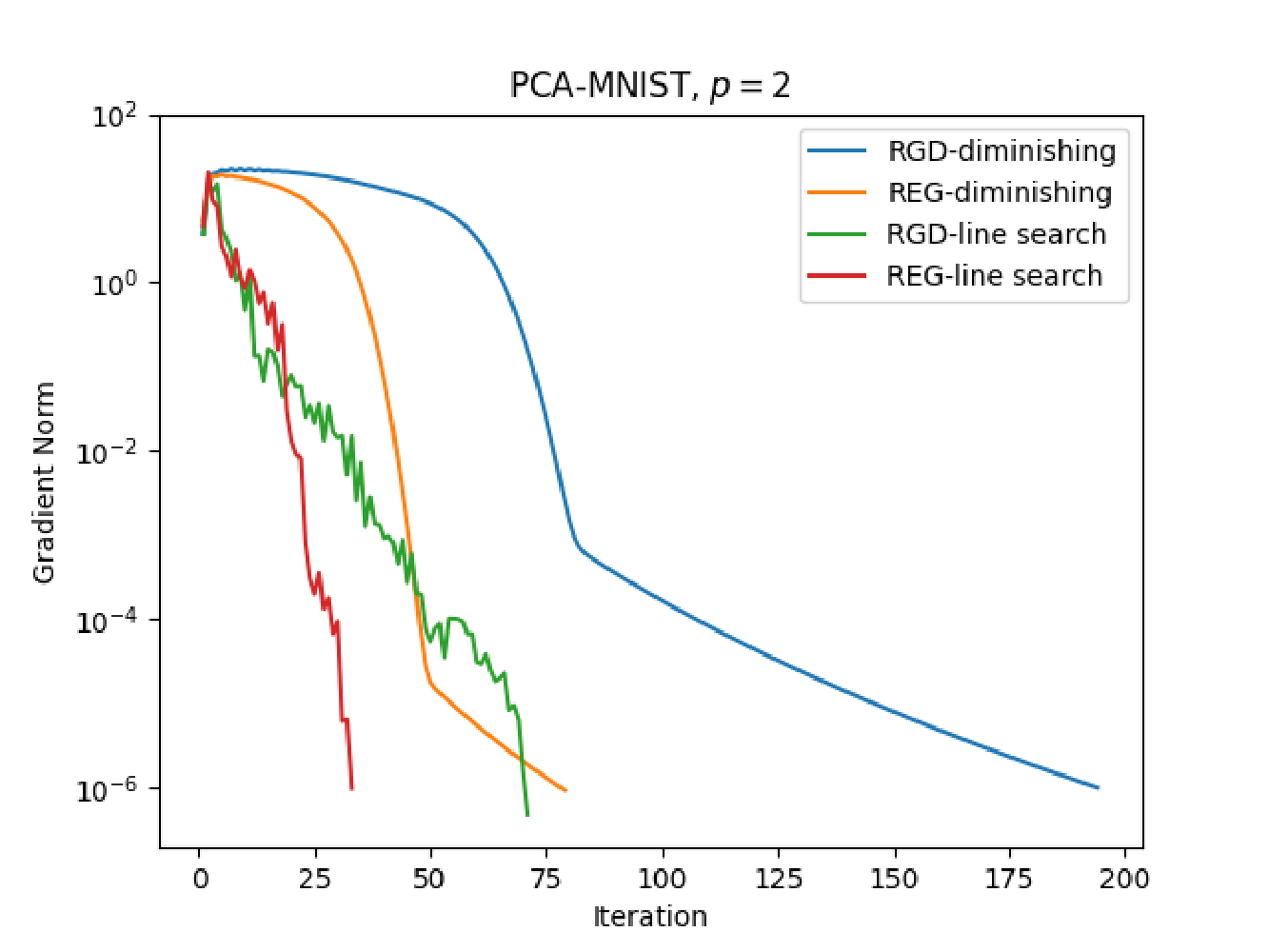}}
\subfigure[]{
\includegraphics[width=4.3cm]{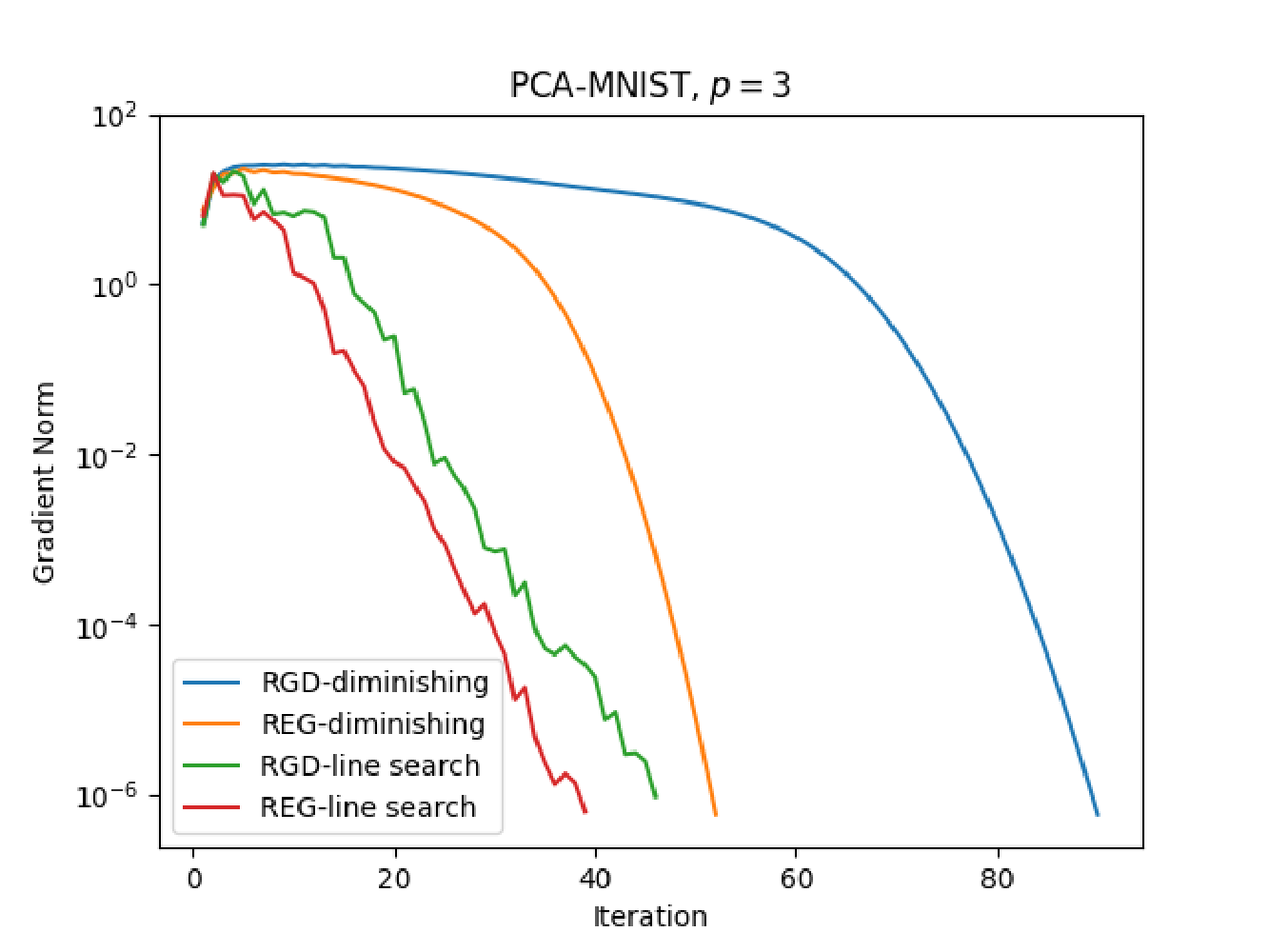}}
\subfigure[]{
\includegraphics[width=4.3cm]{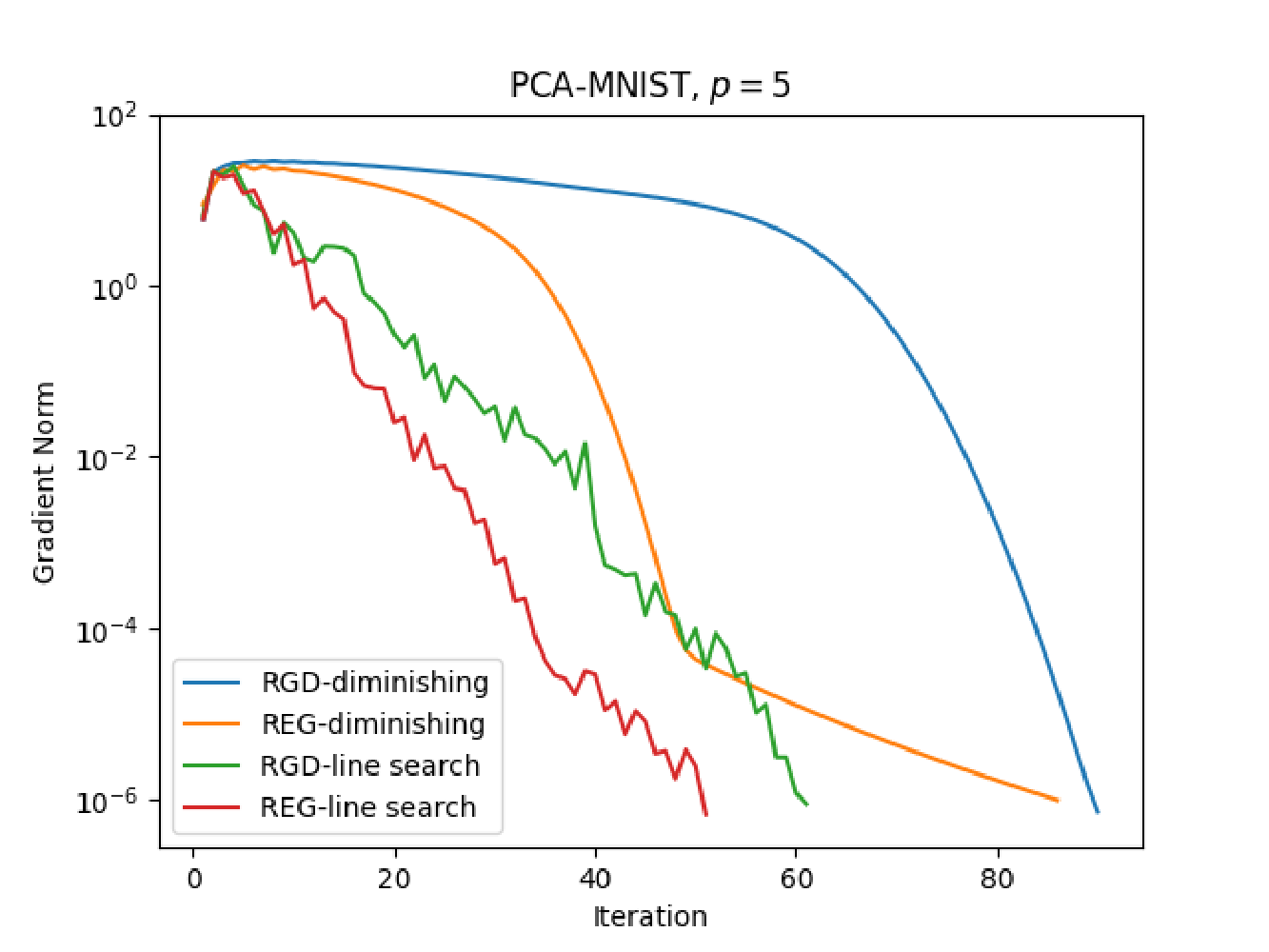}}
\caption{The gradient norm for RGD and REG ($\rho=10^{-2}$) on PCA-MNIST problem.}
\label{fig_mnist}
\end{figure}

From Table \ref{tab_pca}, we compare the performance of RGD and REG using different descent stepsizes $\rho$. The results show that with $\rho=10^{-3}$, REG significantly outperforms RGD, especially when $n=200,p=100$. Additionally, REG based on line search proves to be more effective for PCA problems than using a decreasing stepsize. As illustrated in Fig. \ref{fig_pca}, the gradient norm of REG with $\rho=10^{-3}$ declines at a faster rate than that of RGD. Since the extragradient method takes into account previous gradient information, it can better escape saddle points and local minima, enabling the algorithm to find the global optimal solution more efficiently. These findings highlight that REG is a more effective algorithm overall.

The above conclusion was based on experiments conducted using a randomly generated matrix $A$. To further validate these results, we extended our tests to the real MNIST dataset \cite{lecun1998gradient} to evaluate the performance of the REG algorithm. As shown in Table \ref{tab_mnist}, the performance of REG improves as $\rho$ decreases. For instance, when $\rho=10^{-2}$, Fig. \ref{fig_mnist} illustrates that the gradient descent in REG accelerates more rapidly. These findings further confirm that the inexact gradient algorithm REG effectively overcomes the slow convergence typically observed in traditional gradient methods, demonstrating superior performance.

\section{Conclusion}\label{sec conclusion}
In this paper, we propose a framework for the inexact Riemannian gradient algorithm IRGD and its standardized variants IRGDr, and provide a fundamental convergence analysis of the IRGD based on the Riemannian KL property. Our analysis is conducted in a deterministic setting under standard assumptions. Additionally, we extend our framework to include two applications: RSAM and REG, demonstrating that their convergence results align with the IRGD framework. Numerical experiments on MC and PCA problems validate our analysis and show that IRGD implementations are efficient in practice, with REG performing particularly well.

As a topic for future research, it would be interesting to explore the IRGD framework in a stochastic setting. Additionally, investigating stricter inexact gradient conditions, such as imposing a lower bound on the gradient norm, could be valuable.

\bibliographystyle{siamplain}
\bibliography{ref}

% \appendix
% \section{Useful results}
% Assume the compact smooth submanifold $\Mcal$ is $R$-proximally smooth. 

% \section{Proofs} \label{sec:proof}

\end{document}